\documentclass[11pt,reqno]{amsart}
\usepackage{amsfonts,amstext}
\usepackage{amsmath,amstext,amssymb,amscd,mathrsfs,amsthm}
\usepackage[shortlabels]{enumitem}
\usepackage[dvipsnames,usenames]{xcolor}

\newtheorem{thm}{Theorem}[section]
\newtheorem{lem}[thm]{Lemma}

\newtheorem{prop}[thm]{Proposition}
\newtheorem{assumption}[thm]{Assumption}
\theoremstyle{definition}
\newtheorem{defn}[thm]{Definition}

\theoremstyle{remark}
\newtheorem{rmk}[thm]{Remark}
\numberwithin{equation}{section}

\usepackage[allbordercolors={1 1 1}]{hyperref}

\usepackage{comment}

\usepackage[disable]{todonotes}

\def\grad{\nabla}

\renewcommand{\H}[1][1]{W^{#1,2}(\Omega)}

\newcommand{\R}{\mathbb{R}}
\newcommand{\N}{\mathbb{N}}

\newcommand{\D}{\mathcal{D}}
\newcommand{\A}{\mathscr{A}}

\def\g{\gamma}
\def\l{\lambda}
\def\o{\omega}

\def\t{\tau}

\def\G{\Gamma}
\def\O{\Omega}

\def\A{\mathscr{A}}

\newcommand{\norm}[1]{\left\|#1\right\|}
\newcommand{\abs}[1]{\left|#1\right|}

\def\N{\mathbb{N}}
\def\lra{\longrightarrow}
\def\<{\langle}
\def\>{\rangle}
\def\lra{\longrightarrow}
\def\ra{\rightarrow}
\def\H{H^1_{\G_0}}

\begin{document}

\author[A. R. Becklin]{Andrew R. Becklin}
\address{Department of Mathematics and Computer Science, Drake University, Des Moines, IA  50311, USA} \email{andrew.becklin@drake.edu}

\author[Y. Guo]{Yanqiu Guo}
\address{Department of Mathematics and Statistics, Florida International University, Miami, FL 11101, USA} \email{yanguo@fiu.edu}

\title[A structural acoustic model]{Strong and weak solutions to a structural acoustic model with a $C^1$ source term on the plate}

\date{\today}

\begin{abstract}
In this manuscript, we consider a structural acoustic model consisting of a wave equation defined in a bounded domain $\Omega \subset \mathbb{R}^3$, strongly coupled with a Berger plate equation acting on the flat portion of the boundary of $\Omega$. The system is influenced by an arbitrary $C^1$ nonlinear source term in the plate equation. Using nonlinear semigroup theory and monotone operator theory, we establish the well-posedness of both local strong and weak solutions, along with conditions for global existence. With additional assumptions on the source term, we examine the Nehari manifold and establish the global existence of potential well solutions. Our primary objective is to characterize regimes in which the system remains globally well-posed despite arbitrary growth of the source term and the absence of damping mechanisms to stabilize the dynamics.
\end{abstract}

\maketitle

\section{Introduction}\label{S1}

\subsection{The Model}
In this paper, we study a structural acoustic model influenced by a nonlinear source term $h(w)$ on the plate. Precisely, we study the coupled system of PDEs:
\begin{align}\label{PDE}
\begin{cases}
u_{tt}-\Delta u =0 &\text{ in } \O \times (0,T),\\[1mm]
w_{tt}+\Delta^2w+u_t|_{\G}=h(w)&\text{ in }\G\times(0,T),\\[1mm]
u=0&\text{ on }\G_0\times(0,T),\\[1mm]
\partial_\nu u=w_t&\text{ on }\G\times(0,T),\\[1mm]
w=\partial_{\nu_\G}w=0&\text{ on }\partial\G\times(0,T),\\[1mm]
(u(0),u_t(0))=(u_0,u_1),\hspace{5mm}(w(0),w_t(0))=(w_0,w_1),
\end{cases}
\end{align}
where the  initial data reside in the finite energy space, i.e., 
$$(u_0, u_1)\in H^1_{\G_0}(\O) \times L^2(\O) \, \text{ and }(w_0, w_1)\in H^2_0(\G)\times L^2(\G).$$

In this model, $\O \subset \R^3$ is a bounded, open, connected domain with smooth boundary $\partial \O = \overline{\G_0 \cup \G}$, where $\G_0$ and $\G$ are two disjoint, open, connected sets of positive Lebesgue measure. The set $\G$ is the \emph{flat} portion of the boundary of $\O$ and is referred to as the elastic wall, while $\G_0$ represents the rigid wall. The nonlinearity $h(w)$ represents a source term acting on the plate equation; it is allowed to have a ``bad'' sign and may therefore cause instability. The vectors $\nu$ and $\nu_\G$ denote the outer normals to $\G$ and $\partial \G$, respectively.

Models such as \eqref{PDE} arise in the context of modeling gas pressure in an 
acoustic chamber $\O$ which is surrounded by a combination of rigid and 
flexible walls. The pressure in the chamber is described by the solution 
to a wave equation, while vibrations of the flexible wall are described 
by the solution to a Berger plate equation.  We refer the reader to \cite{Chueshov-1999} and the references quoted therein for more details on the Berger model.  The coupling of the wave and plate equations occurs through the term $u_t|_{\G}$.  This acoustic pressure term can often create unique challenges outside of those generally faced in the study of nonlinear hyperbolic equations.  We refer the reader to work by Beale \cite{Beale76} for a derivation of this term.

The necessity to examine models containing nonlinearities lies in their ability to allow for refinements on weaker responses from equilibrium.  More generally, they enable one to capture coupling of the model with other types of dynamics.  Oftentimes the implementation of nonlinear source terms comes with some amount of damping terms in order to control the behavior of solutions, although none are present in \eqref{PDE}. In light of the significance of studying nonlinear source terms, we are primarily interested in treating $h$ in great generality, requiring only that it be a $C^1$ function in this paper.

\subsection{Literature Overview}
Structural acoustic interaction models have a rich history. These models are well known in both the physical and mathematical literature and date to the canonical models considered in \cite{Beale76,Howe1998}. In the context of stabilization and controllability of structural acoustic models, there is a large body of literature, and we refer the reader to the monograph by Lasiecka \cite{Las2002} which provides a comprehensive overview and quotes many works on these topics. Other related contributions worthy of mention include \cite{Avalos2,Avalos1,Avalos3,Avalos4,Cagnol1,MG1,MG3,LAS1999}.

Previous work by Becklin and Rammaha \cite{Becklin-Rammaha2} studied a model similar to \eqref{PDE}, where nonlinear source and damping terms appeared in both the wave and plate equations, and they established well-posedness criteria. Feng, Guo, and Rammaha \cite{FengGuoRammaha1,FengGuoRammaha2} built on Becklin and Rammaha \cite{Becklin-Rammaha2} to establish potential well solutions, energy decay rates, and blow-up of solutions under various scenarios.

We utilize monotone operator theory and nonlinear semigroups to establish local well-posedness of strong and weak solutions to system \eqref{PDE}, in particular implementing Kato's Theorem as in \cite{Barbu2,Sh}. The novelty of this manuscript lies in the limited assumptions imposed on the term $h$, as well as the absence of any damping term on the plate.  
We establish the local well-posedness of strong and weak solutions for \eqref{PDE} under the sole assumption that $h$ is a $C^1$ function.  
Moreover, we show that if $|h(s)| \leq c(|s|+1)$, then both strong and weak solutions exist globally.

This manuscript also considers potential well solutions of system \eqref{PDE}. 
The study of such solutions for nonlinear hyperbolic equations has a long history. 
For example, in their classical work \cite{PS}, Payne and Sattinger studied potential well solutions of a nonlinear hyperbolic equation in the canonical form
\begin{align}  \label{canonical}
u_{tt} = \Delta u + f(u), \;\; \text{with} \;\; u(0)=u_0, \;\; u_t(0)=u_1,
\end{align}
subject to the boundary condition $u=0$ on the boundary of a bounded domain $\Omega \subset \mathbb{R}^n$.  
In the literature, power-type nonlinear terms are usually considered when applying potential well theory to hyperbolic PDEs. 
Previous work involving non-power source terms was carried out by Alves and Cavalcanti \cite{AlvesCavalcanti}, 
who studied an exponential source in a two-dimensional wave equation and established potential well solutions, 
building on the work of Ambrosetti and Rabinowitz \cite{AMBROSETTI1973349}. See also \cite{MaTF, Minh} for studies of nonlinear hyperbolic equations with exponential nonlinearities.
In this manuscript, by taking advantage of the additional regularity of solutions provided by the plate, 
we consider a more general $C^1$ source term $h(w)$ satisfying Assumption \ref{assumption3}, 
and we establish the global existence of potential well solutions in both the weak and strong sense.

\subsection{Structure of the Paper}

The manuscript is organized as follows. In Section~\ref{sec-results}, we state the main results of the paper concerning system~(\ref{PDE}).  
In Section~\ref{sec-strong}, we establish the local and global well-posedness of strong solutions.  
In Section~\ref{sec-weak}, we justify the local and global well-posedness of weak solutions.  
In Section~\ref{sec-potential}, we prove the global existence of solutions within the potential well framework.

\vspace{0.1 in}

\section{Preliminaries and Main Results} \label{sec-results}
\subsection{Notation}

Throughout the paper, the following notational conventions for $L^p$ space norms and standard inner products will be used:
\begin{align*}
&||u||_p=||u||_{L^p(\O)}, &&(u,v)_\O = (u,v)_{L^2(\O)},\\
&|u|_p=||u||_{L^p(\G)},&&(u,v)_\G = (u,v)_{L^2(\G)}.
\end{align*}
We use both $\g u$ and $u_{|_\G}$ to denote the \emph{trace} of $u$ on $\G$.  We write $\frac{d}{dt}(\g u(t))$ as $\g u_t$.  As is customary, $C$ shall always denote a positive constant which may change from line to line.

We put
$$\H(\O):=\{u\in H^1(\O):u|_{\G_0}=0\}.$$ 

Furthermore, we adopt the standard convention of using Poincar\'{e}'s Inequality to identify the norm $\|u\|_{H^1_{\G_0}}$ with $\|\nabla u\|_2$.  In this we put:
$$\norm{u}_{\H(\O)}= \norm{\grad u}_2,\,\, (u,v)_{\H(\O)}=(\grad u,\grad v)_\O.$$
Similarly, we put:
$$\norm{w}_{H_0^2(\G)}= \abs{\Delta w}_2, \,\, (w,z)_{H^2_0(\G)}=(\Delta w,\Delta z)_\G.$$
For convenience and brevity, we shall  frequently use the notation:
\[
\norm{u}_{1, \O}= \norm{\grad u}_2,\,\,\, \norm{w}_{2, \G}= \abs{\Delta w}_2.
\]
With $Y$ a Banach space, we denote the duality pairing between the dual space $Y'$ and $Y$ by 
$\<\psi,y\>_{Y',Y}$, or simply by  $\<\cdot,\cdot\>$ when the context is clear. That is, 
\begin{align*}
\< \psi,y \> =\psi(y)\text{ for }y\in Y,\, \psi \in Y'.
\end{align*}

\vspace{0.1 in}

\subsection{Assumptions}
Throughout the paper, we study \eqref{PDE} under the following assumption.

\begin{assumption}\label{assumption1} 
$h: \mathbb R \rightarrow \mathbb R$ is a function in $C^1(\R)$.
\end{assumption}

Under Assumption \ref{assumption1}, we obtain the following result.  
\begin{prop}\label{hprop}  
Suppose $h \in C^1(\mathbb{R})$. Let $\G \subset \R^2$ be a bounded open domain with a $C^1$ boundary.  
Then the Nemytskii operator $h$ is locally Lipschitz from $H^{1+\epsilon}_0(\G)$ into $L^2(\G)$ for any $\epsilon > 0$.  
\end{prop}

\begin{proof}
Assume $\|w_1\|_{H^{1+\epsilon}_0(\G)},  \|w_2\|_{H^{1+\epsilon}_0(\G)} \leq R$. 
Due to the imbedding $H^{1+\epsilon}_0(\Gamma) \hookrightarrow L^{\infty}(\G)$ in two dimensions,
we have $|w_1|_{\infty},  |w_2|_{\infty} \leq CR$. By the mean value theorem, we have 
\begin{align*}
|h(w_1) - h(w_2)|_2^2 
&= \int_{\Gamma} |h(w_1) - h(w_2)|^2 d\Gamma \notag\\
&\leq \int_{\Gamma} \Big[  \max_{|s| \leq C R}        |h'(s)| \Big]^2 |w_1 - w_2|^2 d\Gamma \notag\\
&\leq  C(R) |w_1-w_2|_2^2 \notag\\
&\leq  C(R) \|w_1-w_2\|_{H^{1+\epsilon}_0(\Gamma)}^2. 
\end{align*}
\end{proof}

\vspace{0.1 in}

\subsection{Well-posedness of Strong Solutions}

To state our results regarding the local and global well-posedness of strong solutions, we need to introduce the necessary framework to present \eqref{PDE} in operator-theoretic form. Such a setup is essential for applying the nonlinear semigroup technique.

First, we introduce the Dirichl\'et-Neumann Laplacian, given by: $ A=-\Delta:\D(A)\subset L^2(\O)\lra L^2(\O)$, with its domain $\D(A)=\{u\in H^2(\O):u|_{\G_0}=0,\,\, \partial_\nu u|_\G=0\}$.  The operator $A$ can be extended to a continuous map $A:H^1_{\G_0}(\O)\lra(H^1_{\G_0}(\O))'$, where:
\begin{align}\label{Adp}
\<Au,\phi\>=\int_\O\nabla u\cdot\nabla\phi \, dx=(\nabla u,\nabla\phi)_{\O},
\end{align}
for all  $u,\phi\in H^1_{\G_0}(\O)$.  

We next define the  Dirichl\'et-Neumann map: $R:H^s(\G)\lra H^{s+\frac{3}{2}}(\O) \cap H^1_{\G_0}(\O)$ where $s\geq0$ by: 
\begin{align}\label{2.1}
q=Rp\iff q  \text{  is the weak solution of the problem }
\begin{cases}\Delta q=0\text{ in }\O,
\\q=0\text{ on }\G_0, \\
\partial_\nu q=p\text{ on }\G.
\end{cases}  
\end{align}
It is well known (see, for instance Lasiecka and Triggiani \cite{LT3,LT2}) that $R$ is continuous from $H^s(\G)$ to $H^{s+\frac{3}{2}}(\O) \cap H^1_{\G_0}(\O)$, for $s\geq0$.
Let us note here that (\ref{2.1}) and standard computation yield the following useful identity:
\begin{align}\label{PDE1}
\<ARp,\phi\>=(\nabla Rp, \nabla\phi)_\O= (p,\g\phi)_\G,
\end{align}
for all $p\in L^2(\G)$ and $\phi\in H^1_{\G_0}(\O)$.

Additionally, the biharmonic operator $\Delta^2:\D(\Delta^2)\subset L^2(\G)\lra L^2(\G)$ with its domain
$\D(\Delta^2)=H^4(\G)\cap H^2_0(\G)$ can be extended as a continuous mapping from $H^2_0(\G)$ to $H^{-2}(\G)$, where
\begin{align}\label{Bdp}
\<\Delta^2 w,\phi\>=(\Delta w,\Delta\phi)_\G,
\end{align}
for all $w, \phi\in H^2_0(\G)$.  

By using the operators defined above, system \eqref{PDE} can be formally cast as:
\begin{align}\label{PDE2}
\begin{cases}
u_{tt}+A(u-Rw_t)=0 &\text{ in } \O \times (0,T),\\[2mm]
w_{tt}+\Delta^2w+\g u_t=h(w)&\text{ in }\G\times(0,T),\\[2mm]
(u(0),u_t(0))=(u_0,u_1)\in H^1_{\G_0}(\O)\times L^2(\O),\\[2mm]
(w(0),w_t(0))=(w_0,w_1)\in H^2_0(\G)\times L^2(\G).
\end{cases}
\end{align}
Note that the equation \( \partial_\nu u = w_t \) on \( \Gamma \) in system (\ref{PDE}) has been incorporated into the term \( A(u - Rw_t) \).

Next, we introduce the space $H=H^1_{\G_0}(\O)\times H^2_0(\G)\times L^2(\O)\times L^2(\G)$ with the natural norm
\[
\|U\|_H^2= \norm{\grad u}_2^2 + \abs{\Delta w}_2^2 + \norm{y}_2^2 + \abs{z}_2^2,   \text{ for all } U=(u,w,y,z) \in H,
\]
and define the \emph{nonlinear} operator
\begin{align*}
\mathscr{A}:\D(\mathscr{A})\subset H\lra H
\end{align*}
by
\begin{align}\label{PDE3}
\mathscr{A}\begin{bmatrix}u\\w\\y\\z\end{bmatrix}^{tr}=\begin{bmatrix}-y\\-z\\A(u-Rz)\\\Delta^2w+\g y-h(w)\end{bmatrix}^{tr}
\end{align}
where
\begin{align*}
\D(\mathscr{A})=\Big\{(u,w,y,z)  \in & \left(H^1_{\G_0}(\O)  \times H^2_0(\G)\right)^2 :A(u-Rz)\in L^2(\O), \; \Delta^2w\in L^2(\G)\Big\}.
\end{align*}
Observe that the nonlinear term $h(w) \in L^2(\Gamma)$ because $w\in H^2_0(\Gamma) \hookrightarrow L^{\infty}(\Gamma)$ and $h\in C^1(\mathbb R)$.

With $U = (u, w, u_t, w_t)$, system~\eqref{PDE2} can be written in the operator-theoretic form:
\begin{align}\label{PDE4}
    U_t + \mathscr{A} U = 0, 
    \quad U(0)=U_0 =(u_0, w_0, u_1, w_1) \in H.
\end{align}

\begin{thm}\label{thm:strong} {\bf  (Local and global well-posedness of strong solutions)} Assume $h \in C^1(\mathbb R)$ and $U_0\in\D(\mathscr{A})$. Then equation \eqref{PDE4} has a unique local strong solution $U\in W^{1,\infty}(0,T_0;H)$ for some $T_0>0$, satisfying $U(t)\in\D(\mathscr{A})$ for all $t\in [0,T_0]$, where $T_0>0$ depends on the initial quadratic energy $E(0)$, where the quadratic energy is defined as
	\begin{align*}E(t)&:=\frac{1}{2}\left(\|u_t(t)\|_2^2+\|\nabla u(t)\|_2^2+|w_t(t)|_2^2+|\Delta w(t)|_2^2\right).
	\end{align*}
Furthermore, the following energy identity holds for all $t\in[0,T_0]$:
	\begin{align}\label{energy}
	E(t)=E(0)+\int_0^t\int_\G h(w)w_td\G d\t.
	\end{align}
If in addition, we assume $|h(s)|\leq c(|s|+1)$ for all $s\in \mathbb R$ where $c>0$ is a constant, then the local strong solution can be extended to a global strong solution and $T_0$ can be taken arbitrarily large.
\end{thm}

\smallskip

\begin{rmk}
Theorem \ref{thm:strong} establishes the well-posedness of strong solutions $U$ to equation (\ref{PDE4}). 
Since system (\ref{PDE}) can be reformulated in the abstract form (\ref{PDE4}), $U$ is also a strong solution to system (\ref{PDE}). 
Therefore, Theorem \ref{thm:strong} establishes the well-posedness of strong solutions to system (\ref{PDE}) as well. 
\end{rmk}

\vspace{0.1 in}

\subsection{Well-posedness of Weak Solutions}
To present our results on local and global well-posedness of weak solutions, we first introduce the definition of a suitable weak solution for \eqref{PDE}.
\begin{defn}\label{def:weaksln}
A pair of functions $(u,w)$ is said to be a \emph{weak solution} of \eqref{PDE} on the interval $[0,T]$ provided:
	\begin{enumerate}[(i)]
		\setlength{\itemsep}{5pt}
		\item\label{def-u} $u\in C([0,T];H^1_{\G_0}(\O))$, $u_t\in C([0,T];L^2(\O))$,
		\item\label{def-v} $w\in C([0,T];H^2_0(\G))$, $w_t\in C([0,T];L^2(\G))$,
		\item\label{def-uic} $(u(0),u_t(0))=(u_0,u_1) \in H^1_{\G_0}(\O)\times L^2(\O)$,
		\item\label{def-wic} $(w(0),w_t(0))=(w_0,w_1) \in H^2_0(\G)\times L^2(\G)$,
		\item\label{def-ws} The functions $u$ and $w$ satisfy the following variational  identities 
		for all $t\in[0,T]$: 
\begin{align}\label{wkslnwave}
(u_{t}(t),  \phi(t))_\O & - (u_1,\phi(0))_\O-\int_0^t ( u_t(\tau), \phi_t(\tau) )_\O d\tau
+\int_0^t (\nabla u(\tau), \nabla\phi(\tau) )_\O d\tau \notag \\
&-\int_0^t  (w_t(\tau), \g\phi(\tau) )_\G d\tau=0,
\end{align}
and
\begin{align}\label{wkslnplt}
(w_t(t) & + \g u(t),\psi(t) )_\G  -(w_1 +\g u(0) ,\psi(0))_\G -\int_0^t (w_t(\tau), \psi_t(\tau) )_\G d\tau \notag \\
& -\int_0^t (\g u(\tau), \psi_t(\tau) )_\G d\tau+\int_0^t (\Delta w(\tau), \Delta\psi(\tau) )_\G d\tau \notag \\
&=\int_0^t\int_{\G}h(w(\tau))\psi(\tau) d\G d\tau,
\end{align}
for all test functions $\phi$ and $\psi$ satisfying:
$\phi\in C([0,T];H^1_{\G_0}(\O))$, \\$\psi\in C\left([0,T];H^2_0(\G)\right)$  with 
$\phi_t\in L^1(0,T;L^2(\O))$, and $\psi_t\in L^{1}(0,T;L^2(\G))$.
	\end{enumerate}
\end{defn}

The next theorem states the existence and uniqueness of local weak solutions to \eqref{PDE}, as well as their continuous dependence on the initial data. Moreover, under an additional assumption on the source term, the local weak solution can be extended to a global weak solution for all time.

\begin{thm}\label{thm:localexist} {\bf (Local and global well-posedness of weak solutions)}
	Assume $h\in C^1(\mathbb R)$. Let $U_0=(u_0,w_0,u_1,w_1)\in H = H^1_{\G_0}(\O)\times H^2_0(\G)\times L^2(\O)\times L^2(\G)$.     
    Then problem \eqref{PDE} possesses a unique local weak solution $(u,w)$ in the sense of Definition \ref{def:weaksln} on $[0,T_0]$, where $T_0>0$ depends on the initial quadratic energy $E(0)$. 
	In addition, the energy identity \eqref{energy} holds for all $t\in[0,T_0]$. Furthermore, the weak solutions depend continuously on initial data, namely, if $U^n_0=(u^n_0,w^n_0,u^n_1,w^n_1) \in H$ is a sequence of initial data such that $U^n_0\lra U_0$ in $H$ as $n\lra\infty$, then the corresponding weak solutions $(u^n,w^n)$ and $(u,w)$ of \eqref{PDE} satisfy:
\begin{align*}
(u^n,w^n,u^n_t,w^n_t)\lra(u,w,u_t,w_t)\text{ in }C([0,T_0];H),\text{ as }n\lra\infty,
\end{align*}
where $T_0>0$ can be chosen independent of $n\in\N$. If in addition, we assume $|h(s)|\leq c(|s|+1)$ for all $s\in \mathbb R$ where $c>0$ is a constant, then the local weak solution is indeed a global weak solution and $T_0$ can be taken arbitrarily large.\end{thm}

\vspace{0.1 in}

\subsection{Potential Well Solutions}
Before our final results can be presented, some preliminary work pertaining to the potential well must be done.  First, we place additional assumptions on $h$.

\begin{assumption}\label{assumption3} Let $h\in C^1(\mathbb R)$ satisfy:
\begin{enumerate}
\item there exists $H(s)$ and $\theta>2$ such that $H'(s)=h(s)$ and $0<\theta H(s)<h(s)s$ for all $s\in \mathbb{R}\setminus\{0\}$,
\item $|h(s)|\leq h(|s|)$,
\item $h(s)$ is non-decreasing on $(0,\infty)$,
\item $\frac{h(s)}{s}\rightarrow 0^+$ as $s\rightarrow 0^+$.
\end{enumerate}
\end{assumption}

\begin{rmk}
Condition (1) is the well-known Ambrosetti-Rabinowitz condition often used in elliptic problems. 
Note that letting $s\rightarrow 0$ in the inequality $0<\theta H(s)<h(s)s$ implies
$H(0) = 0$. Also, by condition (4), we see that $h(0)=0$. Thus, both $H$ and $h$ pass through the origin. An example of a function $h(s)$ satisfying Assumption \ref{assumption3} is $ e^{|s|^q}|s|^p s$, where $p, q>0$. Another example is $ e^{e^{|s|}}|s|^p s$ where $p>0$.
\end{rmk}

Define $V=H^1_{\G_0}(\Omega)\times H^2_0(\G)$ and endow it with the natural norm given by
\begin{align}\label{Xnorm}
\|(u,w)\|^2_V=\|\nabla u\|_2^2+|\Delta w|_2^2.
\end{align}
Next, let $H(s)$ be as in Assumption \ref{assumption3}, and define the nonlinear functional $\mathscr{J}:V\rightarrow\mathbb{R}$ by
\begin{align}\label{potentialenergy}
\mathscr{J}(u,w)=\frac{1}{2}\left(\|\nabla u\|_2^2+|\Delta w|_2^2\right)-\int_{\Gamma} H(w)d\G.
\end{align}
Note that the potential energy of \eqref{PDE} at time $t$ is given by $\mathscr{J}(u(t),w(t))$.  The Frechet derivative of $\mathscr{J}$ at $(u,v)\in V$ is given by
\begin{align}\label{Jprime}
\left\langle \mathscr{J}'(u,w),(\phi,\psi)\right\rangle=\int_\Omega \nabla u\cdot\nabla \phi dx +\int_\G\Delta w\Delta\psi d\Gamma-\int_\G h(w)\psi d\G,
\end{align}
for all $(\phi,\psi)\in V$.  
The Nehari manifold associated with the functional $\mathscr{J}$ is defined as
\begin{align}\label{nehari}
\mathscr{N}:=\left\{(u,w)\in V\setminus\{(0,0)\}:\left\langle \mathscr{J}'(u,w),(u,w)\right\rangle=0\right\}.
\end{align}
Combining \eqref{Jprime} and \eqref{nehari} allows us to write
\begin{align}\label{nehari2}
\mathscr{N}=\left\{(u,w)\in V\setminus\{(0,0)\}:\|\nabla u\|_2^2+|\Delta w|_2^2=\int_\G h(w)wd\G\right\}.
\end{align}
Next, define the potential well associated with $\mathscr{J}(u,w)$ by 
\begin{align}\label{potentialwell}
\mathscr{W}:=\left\{(u,w)\in V:\mathscr{J}(u,w)<d\right\},
\end{align}
where $d$ is the depth of the potential well $\mathscr{W}$ and taken to be 
\begin{align}\label{depth}
d:=\inf_{(u,w)\in\mathscr{N}}\mathscr{J}(u,w).
\end{align}
We must first verify that $d$ is strictly positive in order to ensure $\mathscr{W}$ is non-empty.
\begin{prop}\label{positived} Let $h$ satisfy Assumption \ref{assumption3}.  Then $d>0$.
\end{prop}
\begin{proof}
Let $(u,w)\in\mathscr{N}$. We claim $w\not =0$. Indeed, if $w=0$, then since $\|\nabla u\|_2^2+|\Delta w|_2^2=\int_\G h(w)wd\G$, we obtain $\|\nabla u\|_2 =0$. By the Poincar\'e inequality, it follows that $\|u\|_2 =0$. Hence $(u,w) = (0,0)$, contradicting $(u,w) \in \mathscr N$.

Then we have from \eqref{potentialenergy}, \eqref{nehari2}, and condition (1) of Assumption \ref{assumption3} that
\begin{align}\label{posd1}
\mathscr{J}(u,w)&=\frac{1}{2}\left(\|\nabla u\|_2^2+|\Delta w|_2^2\right)-\int_\G H(w)d\G\notag\\
&>\frac{1}{2}\left(\|\nabla u\|_2^2+|\Delta w|_2^2\right)-\frac{1}{\theta}\int_\G h(w)wd\G\notag\\
&=\left(\frac{1}{2}-\frac{1}{\theta}\right)\left(\|\nabla u\|_2^2+|\Delta w|_2^2\right),
\end{align}
where $\theta>2$. Note that as $(u,w)\in\mathscr{N}$, and using conditions (2)-(3) of Assumption \ref{assumption3}, we have
\begin{align}\label{posd2}
\|\nabla u\|_2^2+|\Delta w|_2^2&=\int_\G h(w)wd\G\notag\\
&\leq |\G|\sup_{x\in\G}|h(w)w|\notag\\
&\leq|\G|\sup_{x\in\G}\left(h(|w|)|w|\right)\notag\\
&\leq|\G|h\left(|w|_\infty\right)|w|_\infty\notag\\
&\leq|\G|h\left(C|\Delta w|_2\right)C|\Delta w|_2,
\end{align}
since $H^2_0(\G)\hookrightarrow L^\infty(\G)$.

Let $y(t):=\left(\|\nabla u\|_2^2+|\Delta w|_2^2\right)^\frac{1}{2}$ and note $y(t)\geq |\Delta w(t)|_2$.  Then from \eqref{posd2} one has
\begin{align}\label{posd3}
y^2(t)\leq |\G| h(Cy(t))Cy(t),
\end{align}
and in turn
\begin{align}\label{posd4}
\frac{1}{C^2 |\G|}\leq \frac{h(Cy(t))}{Cy(t)},
\end{align}
since $y(t)>0$. By condition (4) of Assumption \ref{assumption3}, we can conclude that there is some $c_0>0$ such that $y(t)\geq c_0$.  Combining this with \eqref{posd1} yields
\begin{align}\label{posd5}
\mathscr{J}(u,w)&>\left(\frac{1}{2}-\frac{1}{\theta}\right)\left(\|\nabla u\|_2^2+|\Delta w|_2^2\right)\notag\\
&= \left(\frac{1}{2}-\frac{1}{\theta}\right)y^2(t)\notag\\
&\geq \left(\frac{1}{2}-\frac{1}{\theta}\right)c_0^2>0.
\end{align}
As $c_0$ is independent of the choice of $(u,w)\in\mathscr{N}$, we can conclude that $d\geq \left(\frac{1}{2}-\frac{1}{\theta}\right)c_0^2>0$.
\end{proof}
Note as expected that the potential well and Nehari manifold are disjoint by \eqref{potentialwell} and \eqref{depth}, that is
\begin{align}\label{disjoint}
\mathscr{W}\cap\mathscr{N}=\varnothing.
\end{align}
This enables the decomposition of the potential well into two parts, the `stable' part denoted $\mathscr{W}_1$ and `unstable' part denoted $\mathscr{W}_2$.  Formally we define them as follows:
\begin{align*}
&\mathscr{W}_1:=\left\{(u,w)\in\mathscr{W}:\|\nabla u\|_2^2+|\Delta w|_2^2>\int_\G h(w)wd\G\right\}\cup\{(0,0)\},\\
&\mathscr{W}_2:=\left\{(u,w)\in\mathscr{W}:\|\nabla u\|_2^2+|\Delta w|_2^2<\int_\G h(w)wd\G\right\}.
\end{align*}
Clearly we have $\mathscr{W}_1\cap\mathscr{W}_2=\varnothing$ and $\mathscr{W}_1\cup\mathscr{W}_2=\mathscr{W}$.

Lastly we define the total energy $\mathscr{E}(t)$ of the system \eqref{PDE} as follows:
\begin{align}\label{totalen}
\mathscr{E}(t):&=\frac{1}{2}\left(\|u_t(t)\|_2^2+|w_t(t)|_2^2\right)+\mathscr{J}(u(t),w(t)) =E(t)-\int_{\Gamma} H(w)d\G,
\end{align}
where the quadratic energy $E(t)=\frac{1}{2}\left(\|u_t(t)\|_2^2+\|\nabla u(t)\|_2^2+|w_t(t)|_2^2+|\Delta w(t)|_2^2\right)$. In this manner, one can rewrite the energy identity \eqref{energy} as
\begin{align}\label{modenergy}
\mathscr{E}(t)=\mathscr{E}(0),
\end{align}
for all $t\in[0,T_0]$, and we can conclude that $\mathscr{E}(t)$ is constant on $[0,T_0]$.  We can now state our final results of the paper:

\begin{thm}\label{thm:wellsol} \textbf{(Global existence of potential well solutions-Part 1)}  
Let $h$ satisfy Assumption \ref{assumption3}. Assume that $(u_0, w_0) \in \mathscr{W}_1$ and $\mathscr{E}(0) < d$. Then system \eqref{PDE} admits a global weak solution $(u(t), w(t)) \in \mathscr W_1$ for all $t\geq 0$. Moreover, for all $t \geq 0$, the following estimates hold for the potential energy $\mathscr{J}(u(t), w(t))$, the total energy $\mathscr{E}(t)$, and the quadratic energy $E(t)$:
\begin{align} \label{3est}
\begin{cases}
    \mathscr{J}(u(t), w(t)) \leq \mathscr{E}(t) = \mathscr{E}(0) < d, \\
    E(t) < \dfrac{\theta d}{\theta - 2},
\end{cases}
\end{align}
where $\theta > 2$ is the constant given in Assumption \ref{assumption3}.
\end{thm}

By assuming additional regularity of the initial data, we obtain the global existence of strong solutions in the potential well. 

\begin{thm}\label{thm:wellsol2} \textbf{(Global existence of potential well solutions-Part 2)}  
Let $h$ satisfy Assumption \ref{assumption3}. Assume that $(u_0, w_0) \in \mathscr{W}_1$ and $\mathscr{E}(0) < d$. 
Suppose further that $U_0 = (u_0,w_0,u_1,w_1)\in\D(\mathscr{A})$. Then equation \eqref{PDE4} admits a unique global strong solution $U = (u,w, u_t, w_t)\in W^{1,\infty}(0,T;H)$, satisfying $U(t)\in\D(\mathscr{A})$ and $(u(t), w(t)) \in \mathscr W_1$ for all $t\in [0,T]$, where $T>0$ is arbitrary. 
Moreover, all inequalities in (\ref{3est}) hold. 
\end{thm}

\vspace{0.1 in}

\section{Well-posedness of Strong Solutions} \label{sec-strong}
This section is devoted to proving Theorem~\ref{thm:strong}, which establishes the local and global well-posedness of strong solutions. 
We employ the theory of nonlinear semigroups and monotone operators. The first step was to reformulate the PDE system~\eqref{PDE} in the operator-theoretic form~\eqref{PDE4}, as demonstrated in the previous section. Due to the presence of the nonlinear source term $h(w)$ acting on the plate, our strategy involves first addressing the case where $h$ is globally Lipschitz from $H^2_0(\Gamma)$ to $L^2(\Gamma)$, and subsequently extending the analysis to the locally Lipschitz scenario.

\subsection{Globally Lipschitz Sources}
This step addresses the case where the source term $h: H^2_0(\Gamma) \rightarrow L^2(\Gamma)$ is globally Lipschitz. Under this assumption, we establish the following lemma.

\begin{lem}\label{Lem:GLS} 
Assume that $h:H^2_0(\Gamma)\to L^2(\Gamma)$ is globally Lipschitz. For any initial data $U_0\in\D(\mathscr{A})$, equation \eqref{PDE4} admits a unique global strong solution $U\in W^{1,\infty}(0,T;H)$ for arbitrary $T>0$, with $U(t)\in\D(\mathscr{A})$ for all $t\in [0,T]$.
\end{lem}

\begin{proof}Recall that by Kato's Theorem (see, for example, \cite{Sh}), if $\mathscr{A}+\o I$ is $m$-accretive for some positive $\o$, then for each $U_0\in\D(\mathscr{A})$ there is a unique strong solution $U$ of \eqref{PDE4}, i.e., $U\in W^{1,\infty}(0,T;H)$ such that $U(0)=U_0$, $U(t)\in\D(\mathscr{A})$ for all $t\in[0,T]$, and equation \eqref{PDE4} is satisfied a.e. on $[0,T]$, where $T>0$ is arbitrary.  In this light, it suffices to show that the operator $\mathscr{A}+\o I$ is $m$-accretive for some positive $\o$.  Recall that an operator $\mathscr{A}:\D(\mathscr{A})\subset H\lra H$ is accretive if $(\mathscr{A}x_1-\mathscr{A}x_2,x_1-x_2)_H\geq0$, for all $x_1,x_2\in\D(\mathscr{A})$, and it is $m$-accretive if, in addition, $\mathscr{A}+I$ maps $\D(\mathscr{A})$ onto $H$.\\

\noindent{}\textbf{Step 1: Accretivity of $\mathscr{A}+\omega I$ for some positive $\omega$.}
Let $U=(u,w,y,z)$, $\hat{U}=(\hat{u},\hat{w},\hat{y},\hat{z}) \in\D(\mathscr{A})$.  We aim to find $\o>0$ such that
\begin{align*}
((\mathscr{A}+\o I)U-(\A+\o I)\hat{U},U-\hat{U})_H\geq0.
\end{align*}
By straightforward calculations, we obtain:
\begin{align}\label{GLS1}
((\mathscr{A}&+\o I)U-(\A+\o I)\hat{U},U-\hat{U})_H\notag\\
=&(\A (U)-\A( \hat{U}),U-\hat{U})_H+\o|U-\hat{U}|_H^2\notag\\
=&-(\grad(y-\hat{y}), \grad(u-\hat{u}))_{\O}-(\Delta (z-\hat{z}),\Delta (w-\hat{w}))_{\G}\notag\\
&+\<A(u-\hat{u}),y-\hat{y}\>-\<AR(z-\hat{z}),y-\hat{y}\>+\< \Delta^2(w-\hat{w}),z-\hat{z}\>
\notag\\ 
&+(\g(y-\hat{y}),z-\hat{z})_\G -(h(w)-h(\hat{w}),z-\hat{z})_\G+\o|U-\hat{U}|_H^2.
\end{align}
Using \eqref{Adp}, \eqref{PDE1}, and \eqref{Bdp}, we have
\begin{align}\label{GLS2}
\begin{cases}
\<A(u-\hat{u}),y-\hat{y}\>=(\nabla(u-\hat{u}),\nabla(y-\hat{y}))_{\O}, \vspace{.1in} \\
\<AR(z-\hat{z}),y-\hat{y}\>=(z-\hat{z},\g(y-\hat{y}))_\G, \vspace{.1in}\\
\< \Delta^2(w-\hat{w}),z-\hat{z}\>= (\Delta(w-\hat{w}),\Delta(z-\hat{z}))_\G.
\end{cases}
\end{align}
Since $h:H^2_0(\G)\lra L^2(\G)$ is globally Lipschitz, we let $L$ be the Lipschitz constant of $h$. Then, by Young's inequality,
\begin{align}\label{GLS7}
-(h(w)-h(\hat{w}),z-\hat{z})_\G&\geq-L | \Delta (w-\hat{w}) |_{2} |z-\hat{z}|_2\notag\\
&\geq-\frac{L}{2}\left(| \Delta (w-\hat{w}) |_{2}^2+|z-\hat{z}|_2^2\right)\notag\\
&\geq-\frac{L}{2}|U-\hat{U}|_H^2.
\end{align}
Combining \eqref{GLS1}-\eqref{GLS7}, we find
\begin{align*}
((\mathscr{A}&+\o I)U-(\A+\o I)\hat{U},U-\hat{U})_H\geq\left(\o-\frac{L}{2}\right)|U-\hat{U}|_H^2.
\end{align*}
Thus, by selecting $\omega>\frac{L}{2}$, we establish the accretivity of $\mathscr{A}+\omega I$.\\

\noindent{}\textbf{Step 2: $m$-accretivity of $\mathscr A+\l I$ for some $\l>0$.} 
To invoke Kato's Theorem and complete the proof of Lemma \ref{Lem:GLS}, one must establish that $\mathscr{A}+\l I:\D(\A)\lra H$ is onto for some $\l>0$.\\

\noindent{}Let $\l>0$ (to be determined later) and $(a,b,c,d)\in H$.  We must show that there exists $(u,w,y,z)\in\D(\A)$ such that $(\A+\l I)(u,w,y,z)=(a,b,c,d)$, that is, 
\begin{align}\label{GLS8}
\begin{cases}
-y+\l u=a\\
-z+\l w=b\\
A(u-Rz)+\l y=c\\
\Delta^2 w+\g y-h(w) +\l z=d.
\end{cases}
\end{align}
Note, \eqref{GLS8} is equivalent to
\begin{align}\label{GLS9}
\begin{cases}
\frac{1}{\l}Ay-ARz+\l y=c-\frac{1}{\l}Aa \vspace{.1in}\\
\frac{1}{\l}\Delta^2 z+\g y-h(\frac{b+z}{\l})+\l z=d-\frac{1}{\l}\Delta^2b.
\end{cases}
\end{align}
Let $V=H^1_{\G_0}(\O)\times H^2_0(\G)$ and notice that the right hand side of \eqref{GLS9} belongs to $V'$.  We now define the operator $\mathscr{B}:V\lra V'$ by:
\begin{align*}
\mathscr{B}\begin{bmatrix}y \vspace{.1in} \\z\end{bmatrix}^{tr}=\begin{bmatrix}\frac{1}{\l}Ay-ARz+\l y \vspace{.1in}\\ \frac{1}{\l}\Delta^2 z+\g y-h(\frac{b+z}{\l})+\l z\end{bmatrix}^{tr}.
\end{align*}
At this point we wish to establish that $\mathscr{B}$ is surjective.  By Corollary 1.2 (p.45) in Barbu \cite{Barbu2}, it is enough to show that $\mathscr{B}$ is maximal monotone and coercive.  To accomplish this we will first show that $\mathscr{B}$ is strongly monotone.  To this end, let $Y, \, \hat{Y} \in V$, where $Y= (y,z)$ and $\hat{Y}=(\hat{y},\hat{z})$.  By straightforward calculations, we obtain
\begin{align}\label{2.14}
\<\mathscr{B}&(Y)-\mathscr{B}(\hat{Y}),Y-\hat{Y}\>_{V',V}\notag\\
&=\frac{1}{\l}\<A(y-\hat{y}),y-\hat{y}\> -\<AR(z-\hat{z}),y-\hat{y}\>+\l\|y-\hat{y}\|_2^2\notag\\
&+\frac{1}{\l}\<\Delta^2(z-\hat{z}),z-\hat{z}\> +(\g (y-\hat{y}),z-\hat{z})_{\G}\notag\\
&-\left(h\left(\frac{b+z}{\l}\right)-h\left(\frac{b+\hat{z}}{\l}\right),z-\hat{z}\right)_{\G}+\l|z-\hat{z}|_2^2.
\end{align}

Thanks to \eqref{Adp}, \eqref{PDE1}, and \eqref{Bdp}, we have
\begin{align}\label{mmsubs}
\begin{cases}
\<A(y-\hat{y}),y-\hat{y}\>=\|\nabla(y-\hat{y})\|_2^2, \vspace{.1in} \\
\<AR(z-\hat{z}),y-\hat{y}\>=(z-\hat{z},\g(y-\hat{y}))_\G, \vspace{.1in}\\
\< \Delta^2(z-\hat{z}),z-\hat{z}\>= |\Delta(z-\hat{z})|_2^2.
\end{cases}
\end{align}

Note from H\"{o}lder's inequality that
\begin{align*}
\left(h\left(\frac{b+z}{\l}\right)-h\left(\frac{b+\hat{z}}{\l}\right),z-\hat{z}\right)_{\G}\leq \left|h\left(\frac{b+z}{\l}\right)-h\left(\frac{b+\hat{z}}{\l}\right)\right|_2\left|z-\hat{z}\right|_2,
\end{align*}
and, by way of $h$ being globally Lipschitz with Lipschitz constant $L$, we further establish that
\begin{align}\label{hterm}
\left(h\left(\frac{b+z}{\l}\right)-h\left(\frac{b+\hat{z}}{\l}\right),z-\hat{z}\right)_{\G} \leq \frac{L}{\l}|\Delta (z-\hat{z})|_2\left|z-\hat{z}\right|_2.
\end{align}

It follows now from \eqref{2.14}, \eqref{mmsubs}, \eqref{hterm}, and Young's inequality (with an $\eta>0$) that:
\begin{align}\label{2.15}
\<\mathscr{B}&(Y)-\mathscr{B}(\hat{Y}),Y-\hat{Y}\>_{V',V}\notag\\
&\geq\frac{1}{\l}\| \grad (y-\hat{y}) \|_{2}^2+\l\|y-\hat{y}\|_2^2 
+\frac{1}{\l}  | \Delta (z-\hat{z}) |_{2}^2-\frac{L}{\l} | \Delta (z-\hat{z}) |_{2} |z-\hat{z}|_2+\l|z-\hat{z}|_2^2 \notag \\
&\geq\frac{1}{\l}\| \grad (y-\hat{y}) \|_{2}^2+\l\|y-\hat{y}\|_2^2 +\frac{1}{\l}  | \Delta (z-\hat{z}) |_{2}^2\notag\\
&\hspace{5mm}-\frac{L^2}{4 \eta \l} | \Delta (z-\hat{z}) |_{2}^2 - \frac{\eta}{\l} | z-\hat{z} |_{2}^2 + \l |z-\hat{z}|_2^2\notag\\
&\geq\frac{1}{\l}\| \grad (y-\hat{y}) \|_{2}^2+\l\|y-\hat{y}\|_2^2 +\left( \frac{1}{\l}  -\frac{L^2}{4 \eta \l} \right) | \Delta (z-\hat{z}) |_{2}^2 + \left(\l  - \frac{\eta}{\l}\right) | z-\hat{z} |_{2}^2. 
\end{align}
By selecting $\eta = \frac{L^2}{2}$ and then selecting $\l = \frac{L}{\sqrt{2}}$, it follows from \eqref{2.15} that 
\begin{align*}
\<\mathscr{B}(Y)-\mathscr{B}(\hat{Y}),Y-\hat{Y}\>_{V',V} \geq \frac{1}{2 \l} \Big( \| \grad (y-\hat{y}) \|_{2}^2 +| \Delta (z-\hat{z}) |_{2}^2\Big) =\frac{1}{2\lambda} \|Y-\hat{Y}\|^2_V,
\end{align*} 
proving $\mathscr{B}$ is strongly monotone.  It is easy to see that strong monotonicity implies coercivity of $\mathscr{B}$.  

Next, we show that $\mathscr{B}$ is continuous.  It is clear that the mappings $A:H^1_{\G_0}(\O)\lra (H^1_{\G_0}(\O))'$, $\Delta^2:H^2_0(\G)\lra H^{-2}(\G)$, and $\g:H^1_{\G_0}(\O)\lra H^{-2}(\G)$ are continuous.  Moreover, $h:H^2_0(\G)\lra L^2(\G)$ is globally Lipschitz continuous, thus the mapping $ z \mapsto h(\frac{b+z}{\l})$ is also continuous from $H^2_0(\G)$ to $ H^{-2}(\G)$. In addition, by the properties of the  Dirichl\'et-Neumann map $R$, we deduce that $AR:H^2_0(\G)\lra (H^1_{\G_0}(\O))'$ is continuous.  It follows that $\mathscr{B}$ is continuous. Recall again from \cite{Barbu2} that along with the monotonicity established earlier, continuity allows us to conclude that $\mathscr{B}$ is maximal monotone.  With coercivity we finally have that $\mathscr{B}$ is surjective.

As $\mathscr{B}$ is surjective, we have established the existence of $(y,z)\in V=H^1_{\G_0}(\O)\times H^2_0(\G)$ such that $(y,z)$ satisfies \eqref{GLS9}.  In addition, $(u,w)=\left(\frac{a+y}{\l},\frac{b+z}{\l}\right)\in H^1_{\G_0}(\O)\times H^2_0(\G)$.  
Moreover, by (\ref{GLS8}), we have $A(u-Rz) \in L^2(\Omega)$ and $\Delta^2 w \in L^2(\Gamma)$. 
Therefore, $(u,w,y,z)\in\D(\mathscr{A})$ with the property that $(\A+\l I)(u,w,y,z)=(a,b,c,d)$ for $\lambda=\frac{L}{\sqrt{2}}$.  Thus, the proof of maximal accretivity is completed and so is the proof of Lemma \ref{Lem:GLS}.
\end{proof}

\vspace{0.1 in}

\subsection{Locally Lipschitz Sources}
In this subsection, we employ a standard truncation procedure (see, e.g., \cite{Becklin-Rammaha2, BL1, CEL1, GR}) combined with Lemma~\ref{Lem:GLS} to establish a result for the case where $h$ is locally Lipschitz continuous from $H^2_0(\Gamma)$ to $L^2(\Gamma)$.

In Theorem~\ref{thm:strong}, concerning the well-posedness of strong solutions, we assume $h \in C^1(\mathbb{R})$. This implies the local Lipschitz continuity of the Nemytskii operator $h: H^2_0(\Gamma) \to L^2(\Gamma)$, by Proposition~\ref{hprop}. Consequently, the local well-posedness of strong solutions stated in Theorem~\ref{thm:strong} follows directly from Lemma~\ref{Lem:LLS} below.

\begin{lem}\label{Lem:LLS}
Assume that $h:H^2_0(\Gamma)\to L^2(\Gamma)$ is locally Lipschitz continuous. 
For any initial data $U_0\in\D(\mathscr{A})$, equation \eqref{PDE4} admits a unique local strong solution $U\in W^{1,\infty}(0,T_0;H)$ for some $T_0>0$, with $U(t)\in\D(\mathscr{A})$ for all $t\in [0,T_0]$, where $T_0>0$ depends on the initial quadratic energy $E(0)$.
\end{lem}

\begin{proof}

Recall from the previous subsection $V=H^1_{\G_0}(\O)\times H^2_0(\G)$ and define
\begin{align*}
h^K(w)=\begin{cases}h(w)&\text{ if } | \Delta w |_{2}\leq K, \vspace{.05in}\\h\left(\frac{Kw}{  | \Delta w |_{2}}\right)&\text{ if } | \Delta w |_{2} >K,
\end{cases}
\end{align*}
where $K$ is a positive constant such that $K^2\geq 4 E(0)+1$, where the energy $E(t)$ is given by 
\begin{align*}E(t)=\frac{1}{2}\left(\|u_t(t)\|_2^2+|w_t(t)|_2^2 +\|\grad u(t)\|_2^2+ |\Delta w(t)|_2^2\right).
\end{align*}
With the truncated source, we consider the following ($K$) problem:
\begin{align*}
(K)\begin{cases}u_{tt}+A(u-Rw_t)=0 &\text{ in } \O \times (0,T),\\[2mm]
w_{tt}+\Delta^2w+\g u_t=h^K(w)&\text{ in }\G\times(0,T),\\[2mm]
(u(0),u_t(0))=(u_0,u_1)\in H^1_{\G_0}(\O)\times L^2(\O),\\[2mm]
(w(0),w_t(0))=(w_0,w_1)\in H^2_0(\G)\times L^2(\G).
\end{cases}
\end{align*}
We note here that for each such $K$, the operator $h^K:H^2_0(\G)\lra L^2(\G)$ is globally Lipschitz continuous (see, e.g., \cite{CEL1}), and therefore by Lemma \ref{Lem:GLS}, the $(K)$ problem has a unique global strong solution $U_K\in W^{1,\infty}(0,T;H)$ for any $T>0$ provided the initial datum $U_0\in \D(\mathscr{A})$.

The analysis that follows considers one ($K$) problem and thus for cleanliness we express $(u_K(t),w_K(t))$ as $(u(t),w(t))$.  
Since $u_t\in H^1_{\G_0}(\O)$ and $w_t\in H^2_0(\G)$, we may use the multiplier $u_t$ and $w_t$ on the $(K)$ problem and obtain the following energy identity:
\begin{align}\label{LLS1}
&E(t)=E(0)+\int_0^t\int_\G h^K(w)w_td\G d\tau.
\end{align}
Note that we have local Lipschitz constant $L_K$ for the function $h$ on the ball $\{(u,w)\in V:\|(u,w)\|_V\leq K \}$.  Note also that $\left|\Delta\left(\frac{Kw}{|\Delta w|_2}\right)\right|\leq K$ for all $w\in H^2_0(\G)$.  Thus $L_K$ further suffices as a global Lipschitz constant for $h^{K}$.

By H\"{o}lder's and Young's inequalities, we now have:
\begin{align}\label{LLS2}
\int_0^t\int_{\Gamma} & h^K(w)w_t dx d\tau \leq\int_0^t\left|h^K(w)\right|_2|w_t|_2d\tau\notag\\
&\leq\frac{1}{2}\int_0^t|w_t|_2^2d\tau+\frac{1}{2}\int_0^t\left|h^K(w)\right|_2^2d\tau\notag\\
&\leq\frac{1}{2}\int_0^t|w_t|_2^2d\tau + \int_0^t  \left( \left|h^K(w)-h^K(0)\right|_2^2+\left|h^K(0)\right|_2^2\right) d\tau\notag\\
&\leq\frac{1}{2}\int_0^t|w_t|_2^2d\tau + L_K^2\int_0^t\|w\|_{2,\G}^2d\tau + t|h(0)|^2|\G|.
\end{align}
Let $C_0 = |h(0)|^2|\G|$, $C_1=\max\left\{2L_K^2 ,1 \right\}$, and in turn select
\begin{align}\label{LLS9}
T_0=\min\left\{\frac{1}{4C_0},\frac{1}{C_1}\log2\right\}.
\end{align}
By combining \eqref{LLS1}-\eqref{LLS9}, one has
\begin{align}\label{LLS6}
E(t)&\leq E(0)+\frac{1}{2}\int_0^t |w_t|_2^2d\tau + L_K^2\int_0^t\|w\|_{2,\G}^2d\tau + t|h(0)|^2|\G|\notag\\
&\leq E(0)+C_0T_0+C_1\int_0^tE(\tau)d\tau,
\end{align}
for all $t\in[0,T_0]$.  Thus by Gronwall's inequality, one has
\begin{align}\label{LLS8}
E(t)\leq(E(0)+C_0T_0)e^{C_1t}\text{ for all }t\in [0,T_0].
\end{align}
Note our selection $K$ such that $K^2\geq 4E(0)+1$.  Thus it follows from \eqref{LLS8} and (\ref{LLS9}) that
\begin{align}\label{LLS10}
E(t)\leq2\left(E(0)+\frac{1}{4}\right)\leq\frac{K^2}{2},
\end{align}
for all $t\in[0,T_0]$.  This implies that $\|(u(t),w(t))\|_V\leq K$, for all $t\in[0,T_0]$, and therefore $h^K(w)=h(w)$ on the time interval $ [0,T_0]$.  
By the uniqueness of solutions to the ($K$) problem, this solution coincides with the solution of system \eqref{PDE2} for $t\in[0,T_0]$.
From (\ref{LLS9}), we see that $T_0$ depends on $K$, which is chosen to satisfy $K^2\geq 4E(0)+1$. Hence, $T_0$ depends on $E(0)$. This completes the proof of Lemma \ref{Lem:LLS}. Thus, the local well-posedness of strong solutions stated in Theorem~\ref{thm:strong} follows.
\end{proof}

\vspace{0.1 in}

\subsection{Global Existence of Strong Solutions}
In this subsection, we show that if $h$ is bounded by a linear function, i.e., $|h(s)| \leq c(|s|+1)$ for all $s\in \mathbb R$, 
then the local strong solution can be extended to a global strong solution for all $t\geq 0$. 

\begin{prop}\label{prop:ge}
Assume $h\in C^1(\mathbb R)$ such that $|h(s)|\leq c(|s|+1)$ for all $s\in\R$ where $c>0$ is a constant. Let $U=(u,w,u_t,w_t)$ be a strong solution of \eqref{PDE4} on $[0,T]$. Then the quadratic energy $E(t)$ satisfies
\begin{align}\label{GE2}
E(t)\leq\left(E(0)+C t\right)e^{C t}, \text{ for all}\;\; t\in[0,T].
\end{align}
\end{prop}

\begin{proof}
From (\ref{energy}), we have the energy identity: 
\begin{align}\label{GE3}
&E(t)=E(0)+\int_0^t\int_\G h(w)w_td\G d\t,  \text{ for}\; t\in[0,T].
\end{align}
To estimate the source term on the right-hand side of \eqref{GE3}, we employ H\"older's and Young's inequalities as follows:
\begin{align}\label{GE5}
\left|\int_0^t\int_\G h(w)w_td\G d\tau\right|&\leq\frac{1}{2}\int_0^t|h(w)|_2^2 d\tau +\frac{1}{2}\int_0^t|w_t|_{2}^{2}d\tau     \notag\\
&\leq C\int_0^t|w|_2^2d\tau+C t + \int_0^t E(\tau)d\tau     \notag\\
&\leq C\int_0^t E(\tau) d\tau + C  t,
\end{align}
where we have used the Poincar\'e inequality $|w|_2 \leq C|\Delta w|_2$.

It follows from \eqref{GE3} and \eqref{GE5} that, for all $t\in[0,T]$,
\begin{align}\label{GE9}
&E(t)\leq E(0) + C t + C\int_0^t E(\tau) d\tau.
\end{align}

By Gronwall's inequality, we conclude that
\begin{align}\label{GE12}
E(t)\leq\left(E(0)+C t\right)e^{C t}, \text{ for all} \; t\in[0,T].
\end{align}
\end{proof}

We now prove that the local strong solution can be extended to a global strong solution under the assumption that $h\in C^1(\mathbb R)$ such that $|h(s)| \leq c(|s| + 1)$, using Proposition~\ref{prop:ge} and Lemma~\ref{Lem:LLS}.

Fix an arbitrary $T > 0$. Define
\begin{align} \label{def-M}
M := \left(E(0) + C T\right)e^{C T}.
\end{align}
Choose
\begin{align} \label{def-K}
K := (4M + 1)^{1/2},
\end{align}
so that $K^2 = 4M + 1 \geq 4E(0) + 1$. By Lemma~\ref{Lem:LLS}, there exists a unique local strong solution to equation \eqref{PDE4} on $[0, T_0]$, where, due to~\eqref{LLS9},
$T_0=\min\left\{\frac{1}{4C_0},\frac{1}{C_1}\log2\right\}$, with $C_0 = |h(0)|^2|\G|$ and $C_1=\max\left\{2L_K^2 ,1 \right\}$. Note that $T_0$ depends on $K$.

From Proposition~\ref{prop:ge}, we obtain
\begin{align} \label{ext-1}
E(t) \leq \left(E(0) + C T_0\right)e^{C T_0} \leq M, \quad \text{for all } t \in [0, T_0],
\end{align}
provided $T_0 \leq T$, by using~\eqref{def-M}.

Next, we extend the strong solution from $T_0$ to $2T_0$. Using the same $K$ as in~\eqref{def-K}, we have $K^2 = 4M + 1 \geq 4E(T_0) + 1$ by~\eqref{ext-1}. Applying Lemma~\ref{Lem:LLS} again, we extend the solution to $[T_0, 2T_0]$. Repeating this process iteratively, we extend the solution up to time $T$. The key observation is that, at each extension step, the same $K$ is used, and the energy remains bounded by $M$ before time $T$. Since $T$ is arbitrary, the strong solution is global in time. This completes the proof of the global existence of a strong solution under the assumption  that $h\in C^1(\mathbb R)$ such that $|h(s)| \leq c(|s| + 1)$.

In conclusion, we have proved Theorem~\ref{thm:strong}, that is, the local and global well-posedness of strong solutions 
to equation \eqref{PDE4} and to system \eqref{PDE} as well.

\vspace{0.1 in}

\section{Well-posedness of Weak Solutions} \label{sec-weak}

This section is devoted to proving Theorem \ref{thm:localexist}: local and global well-posedness of weak solutions to system (\ref{PDE}).

\subsection{Local Existence of Weak Solutions} \label{sub-local}
Given a prescribed $U_0\in H$, we will establish the existence of a local weak solution. The proof builds on the local existence of strong solutions established in the previous section. 

\noindent{}\textbf{\underline{Step 1: Approximate solutions.}}\vspace{1mm}
\\
\noindent{} 
Let $U_0=(u_0,u_1,w_0,w_1)\in H$. Since the space of test functions $\mathscr{D}(\O)^4$ is dense in $H$, for each $U_0\in H$ there exists a sequence of functions $U^n_0=(u^n_0,u^n_1,w^n_0,w^n_1)\in \mathscr{D}(\O)^4$ such that $U^n_0\to U_0$ in $H$.  
Set $U=(u,w,u_t,w_t)$ and consider the system
\begin{align}\label{ASPL2}
U_t+\mathscr{A}U=0, \quad U(0)=(u^n_0,u^n_1,w^n_0,w^n_1)\in\mathscr{D}(\O)^4.
\end{align}
For each $n$, equation \eqref{ASPL2} has a strong local solution $U^n=(u^n,w^n,u_t^n,w_t^n)\in W^{1,\infty}(0,T_0;H)$ such that $U^n(t)\in\D(\mathscr{A})$ for $t\in [0,T_0]$,
by Theorem \ref{thm:strong}.

Note in the previous line that $T_0$ does not depend on the selection $n$.  This is because in the proof of Lemma \ref{Lem:LLS}, $K$ depends only on the initial energy.  As $U^n_0\lra U_0$ in $H$, we can select $K$ large enough such that $K^2\geq 4E(0)+1$ and $K^2\geq 4E^n(0)+1$ for all $n$, where 
\begin{align*}
E^n(t)=\frac{1}{2}\left(\|u^n_t(t)\|_2^2+|w^n_t(t)|_2^2 +\|\grad u^n(t)\|_2^2+ |\Delta w^n(t)|_2^2\right).
\end{align*}
As $T_0$ in the proof of Lemma \ref{Lem:LLS} depends only on the value of $K$, we indeed can establish a uniform $T_0>0$ for all $n$.

Now, by \eqref{LLS10}, we know $E^n(t)\leq \frac{K^2}{2}$ for all $t\in [0,T_0]$, which implies that,
\begin{align}\label{ASPL3}
\|U^n(t)\|_H^2=\|\nabla u^n(t)\|_2^2+|\Delta w^n(t)|_2^2+\|u^n_t(t)\|_2^2+|w^n_t(t)|_2^2\leq K^2,
\end{align}
for all $t\in[0,T_0]$, for all $n$.  

Lastly, if $\phi$ and $\psi$ satisfy the conditions imposed on test functions in Definition \ref{def:weaksln}, then we can test the approximate system \eqref{ASPL2} with  $\phi$ and $\psi$ to obtain:
\begin{align}\label{ASPL7}
(u_{t}^n(t),  \phi(t))_\O & - (u_t^n(0),\phi(0))_\O-\int_0^t ( u^n_t(\tau), \phi_t(\tau) )_\O d\tau
+\int_0^t (\nabla u^n(\tau), \nabla\phi(\tau) )_\O d\tau \notag \\
&-\int_0^t  (w_t^n(\tau), \g\phi(\tau) )_\G d\tau=0,
\end{align}
and
\begin{align}\label{ASPL8}
(w_t^n(t) & + \g u^n(t),\psi(t) )_\G  -(w_t^n(0) +\g u^n(0) ,\psi(0))_\G -\int_0^t (w_t^n(\tau), \psi_t(\tau) )_\G d\tau \notag \\
& -\int_0^t (\g u^n(\tau), \psi_t(\tau) )_\G d\tau+\int_0^t (\Delta w^n (\tau), \Delta\psi(\tau) )_\G d\tau \notag \\
&=\int_0^t\int_{\G}h(w^n(\tau))\psi(\tau) d\G d\tau,
\end{align}
for all $t\in[0,T_0]$, for all $n$.
\medskip

\noindent{}\textbf{\underline{Step 2: Passage to the limit.}}\vspace{1mm}
\\
We aim to prove that there exists a subsequence of $\{U^n\}$, labeled again as $\{U^n\}$, that converges to a weak solution of problem \eqref{PDE}. In what follows, we focus on passing to the limit in \eqref{ASPL7} and \eqref{ASPL8}.\\
\\
First, we note that \eqref{ASPL3} shows $\{U^n\}$ is bounded in $L^\infty(0,T_0;H)$.  Thus by Alaoglu's Theorem, there exists a subsequence, labeled as $\{U^n\}$, such that
\begin{align}\label{ASPL9}
U^n\lra U\text{ weakly}^*\text{ in }L^\infty(0,T_0;H).
\end{align}
We note here that the imbedding $H^1_{\G_0}(\O)\hookrightarrow H^{1-\epsilon}_{\G_0}(\O)$ is compact. 
Then, by Aubin-Lions-Simon Compactness Theorem, there exists a subsequence, reindexed by $\{ u^n\}$,  such that 
\begin{align}\label{ASPL10}
u^n\lra u\text{ strongly in } C([0,T_0];H^{1-\epsilon}_{\G_0}(\O)).
\end{align}
Similarly, we deduce that there exists a subsequence such that
\begin{align}\label{ASPL11}
w^n\lra w\text{ strongly in } C([0,T_0];H^{2-\epsilon}_0(\G)).
\end{align}
Now since $H^{1-\epsilon}(\O)\hookrightarrow L^2(\G)$ for sufficiently small $\epsilon>0$, it follows from \eqref{ASPL10} that
\begin{align}\label{ASPL11.5}
\g u^n\lra \g u\text{ strongly in } C([0,T_0]; L^2(\G)).
\end{align}

Using the convergence results \eqref{ASPL9}--\eqref{ASPL11.5}, we can pass to the limit in all linear terms of \eqref{ASPL7} and \eqref{ASPL8}. It remains to show the convergence of the nonlinear term:
\begin{align}\label{ASPL38}
\lim_{n\ra\infty}\int_0^t\int_\G h(w^n)\psi d\G d\tau=\int_0^t\int_\G h(w)\psi d\G d\tau.
\end{align}
By Proposition \ref{hprop}, $h:H^{2-\epsilon}_0(\G)\lra L^2(\G)$ is locally Lipschitz.  Since $\psi\in C([0,T_0];H^2_0(\G))$, we can conclude by way of H\"{o}lder's inequality that
\begin{align}\label{ASPL32}
&\int_0^t\int_\G|h(w^n)-h(w)||\psi|d\G d\tau\notag\\
&\hspace{5mm}\leq\left(\int_0^t\int_\G|h(w^n)-h(w)|^2d\G d\tau\right)^{\frac{1}{2}}  \left(\int_0^t\int_\G|\psi|^2d\G d\tau\right)^{\frac{1}{2}}\notag\\
&\hspace{5mm}\leq C(K)\left(\int_0^t\|w^n-w\|_{H^{2-\epsilon}_0(\G)}^{2} d\tau \right)^{\frac{1}{2}} \longrightarrow 0,
\end{align}
where $C(K)$ is a constant depending on $K$, the uniform bound from \eqref{ASPL3}.  
The convergence in (\ref{ASPL32}) is due to \eqref{ASPL11}. This verifies \eqref{ASPL38}.

In summary, we have successfully passed to the limit in \eqref{ASPL7} and \eqref{ASPL8}, thereby establishing the existence of a solution satisfying \eqref{wkslnwave} and \eqref{wkslnplt}.

\medskip

\noindent{}\textbf{\underline{Step 3: Regularity.}}\vspace{1mm}
\\
To finish the existence portion of the proof of Theorem \ref{thm:localexist}, we must establish the desired regularity on the solutions.  To accomplish this, we consider the difference of two solutions to the approximate problem \eqref{ASPL2}, $U^n$ and $U^j$.  For ease of notation, we put $\tilde{u}=u^n-u^j$ and $\tilde{w}=w^n-w^j$.  Note that as $U^n$ and $U^j$ are solutions to \eqref{ASPL2}, we must have that 
\begin{align}\label{appprob}
\begin{cases}\tilde{u}_{tt}-\Delta\tilde{u}=0&       \text{in} \;  \Omega\times(0,T_0),\\
\tilde{w}_{tt}+\Delta^2\tilde{w}+\g \tilde u_t=h(w^n)-h(w^j)&     \text{in} \;  \G\times(0,T_0).
\end{cases}
\end{align}
As $U^n$, $U^j\in W^{1,\infty}(0,T_0;H)$ with $U^n(t)$, $U^j(t)\in\D(\mathscr{A})$, we have $\tilde{u}_t\in W^{1,\infty}(0,T_0;L^2(\O))$ with $\tilde{u}_t(t)\in H^1_{\G_0}(\O)$. 
Thus, we can multiply the first equation in \eqref{appprob} by $\tilde{u}_t(t)$ and integrate over both $\O$ and $(0,T_0)$ to conclude
\begin{align}\label{ASPL18}
&\frac{1}{2}\left(\|\tilde{u}_t(t)\|_2^2+\|\tilde{u}(t)\|_{1,\O}^2\right)-\int_0^t(\tilde{w}_t,\gamma\tilde{u}_t)_\G d\tau=\frac{1}{2}\left(\|\tilde{u}_t(0)\|_2^2+\|\tilde{u}(0)\|_{1,\O}^2\right).
\end{align}
Similarly we have $\tilde{w}_t\in W^{1,\infty}(0,T_0;L^2(\G))$ with $\tilde{w}_t(t)\in H^2_0(\G)$, and thus can multiply the second equation in \eqref{appprob} by $\tilde{w}_t(t)$ and integrate over both $\G$ and $(0,T_0)$ to conclude
\begin{align}\label{ASPL19}
&\frac{1}{2}\left(|\tilde{w}_t(t)|_2^2+\|\tilde{w}(t)\|_{2,\G}^2\right)+\int_0^t(\gamma \tilde{u}_t,\tilde{w}_t)_\G d\tau\notag\\
&\hspace{10mm}=\frac{1}{2}\left(|\tilde{w}_t(0)|_2^2+\|\tilde{w}(0)\|_{2,\G}^2\right)+\int_0^t\int_\G(h(w^n)-h(w^j))\tilde{w}_td\G d\tau.
\end{align}

Define $\tilde{E}(t)=\frac{1}{2}\left(\|\tilde{u}_t(t)\|_2^2 +\|\tilde{u}(t)\|_{1,\O}^2+|\tilde{w}_t(t)|_2^2+\|\tilde{w}(t)\|_{2,\G}^2\right)$ for $t\in[0,T_0]$.  Then from \eqref{ASPL18} and \eqref{ASPL19}, we have
\begin{align}\label{ASPL20}
\tilde{E}(t)&=\tilde{E}(0)+\int_0^t\int_\G(h(w^n)-h(w^j))\tilde{w}_t d\G d\tau\notag\\
&\leq \tilde{E}(0)+\int_0^{T_0} \int_\G|h(w^n)-h(w^j)||\tilde{w}_t|d\G d\tau.
\end{align}

We now wish to show that both terms on the right-hand side of \eqref{ASPL20} converge to 0 as $n,j\ra\infty$.  First, since $\lim_{n\ra\infty}\|u^n_0-u_0\|_{1,\O}=\lim_{n\ra\infty}\|u^n_1-u_1\|_2=\lim_{n\ra\infty}\|w^n_0-w_0\|_{2,\G}=\lim_{n\ra\infty}|w^n_1-w_1|_2=0$, we obtain
\begin{align}\label{ASPL21}
\lim_{n,j\ra\infty}\tilde{E}(0)=0.
\end{align}

To address the second term, we begin by noting that
\begin{align}\label{ASPL29}
&\int_0^{T_0} \int_\G|h(w^n)-h(w^j)||\tilde{w}_t|d\G d\tau\notag\\
&\hspace{5mm}\leq\int_0^{T_0}\int_\G|h(w^n)-h(w)||\tilde{w}_t|d\G d\tau+\int_0^{T_0}\int_\G|h(w)-h(w^j)||\tilde{w}_t|d\G d\tau.
\end{align}
Without loss of generality, we consider the $w^n$ term in the right side of \eqref{ASPL29}.  
As $\tilde{w}_t\in W^{1,\infty}(0,T_0;L^2(\G))$, we can replace $\psi$ in \eqref{ASPL32} with $\tilde{w}_t$ to conclude that 
\begin{align}\label{ASPL30}
&\lim_{n\ra\infty}\int_0^{T_0}\int_\G|h(w^n)-h(w)||\tilde{w}_t|d\G d\tau=0.
\end{align}
Combining \eqref{ASPL29} and \eqref{ASPL30} yields that
\begin{align}\label{ASPL28}
\lim_{n,j\ra\infty}\int_0^{T_0} \int_\G |h(w^n)-h(w^j)||\tilde{w}_t|d\G d\tau=0.
\end{align}
Note now that \eqref{ASPL20}, \eqref{ASPL21}, and \eqref{ASPL28} give us that
\begin{align}\label{ASPLextra2}
\lim_{n,j\ra \infty} \left[\sup_{t\in [0,T_0]}\tilde{E}(t) \right] =0.
\end{align}
Note that \eqref{ASPLextra2} implies that
\begin{align*}
\begin{split}
&\lim_{n,j\ra\infty}\|u^n(t)-u^j(t)\|_{1,\O}^2=\lim_{n,j\ra\infty}\|\tilde{u}(t)\|_{1,\O}^2=0\text{ uniformly in }t\in[0,T_0];\\
&\lim_{n,j\ra\infty}\|u^n_t(t)-u^j_t(t)\|_{2}^2=\lim_{n,j\ra\infty}\|\tilde{u}_t(t)\|_{2}^2=0\text{ uniformly in }t\in[0,T_0];\\
&\lim_{n,j\ra\infty}\|w^n(t)-w^j(t)\|_{2,\G}^2=\lim_{n,j\ra\infty}\|\tilde{w}(t)\|_{2,\G}^2=0\text{ uniformly in }t\in[0,T_0];\\
&\lim_{n,j\ra\infty}|w^n_t(t)-w^j_t(t)|_{2}^2=\lim_{n,j\ra\infty}|\tilde{w}_t(t)|_{2}^2=0\text{ uniformly in }t\in[0,T_0],
\end{split}
\end{align*}
and hence
\begin{align}\label{ASPL44}
\begin{cases}
u^n(t)\lra u(t)\text{ in }H^1_{\G_0}(\O)\text{ uniformly on }[0,T_0],\\
u^n_t(t)\lra u_t(t)\text{ in }L^2(\O)\text{ uniformly on }[0,T_0],  \\
w^n(t)\lra w(t)\text{ in }H^2_{0}(\G)\text{ uniformly on }[0,T_0], \\
w_t^n(t)\lra w_t(t)\text{ in }L^2(\G)\text{ uniformly on }[0,T_0].
\end{cases}
\end{align}
Since $U^n\in W^{1,\infty}(0,T_0;H)$, by \eqref{ASPL44} we conclude
\begin{align*}
&u\in C([0,T_0];H^1_{\G_0}(\O)), \quad u_t\in C([0,T_0];L^2(\O)),\\
&w\in C([0,T_0];H^2_0(\G)), \quad w_t\in C([0,T_0];L^2(\G)).
\end{align*}
Moreover, \eqref{ASPL44} shows $u^n(0)\lra u(0)$ in $H^1_{\G_0}(\O)$.  Since $u^n(0)=u^n_0\lra u_0$ in $H^1_{\G_0}(\O)$, then the initial condition $u(0)=u_0$ holds.  Also, since $u^n_t(0)\lra u_t(0)$ in $L^2(\O)$ and $u^n_t(0)=u^n_1\lra u_1$ in $L^2(\O)$, we obtain $u_t(0)=u_1$. Similarly, we find that $w(0)=w_0$ and $w_t(0)=w_1$.  This completes the proof of the local existence statement in Theorem \ref{thm:localexist} for weak solutions.

\vspace{0.1 in}

\subsection{Energy Identity for Weak Solutions} \label{sub-EI}
\noindent{}In this section, we derive the energy identity for weak solutions stated in Theorem \ref{thm:localexist}.  As $u_t$ and $w_t$ are not regular enough to be used as test functions in \eqref{wkslnwave} and \eqref{wkslnplt} for weak solutions, we shall use the difference quotients $D_hu$ and $D_hw$ and their well-known properties that appeared in \cite{KL} and later in \cite{GR,RS2,Saw}. 

\subsubsection{Properties of the Difference Quotient}
Let $X$ be a Banach space.  For any function $u\in C([0,T];X)$ and $h>0$, we define the symmetric difference quotient by:
\begin{align}\label{DQ1}
D_hu(t)=\frac{u_e(t+h)-u_e(t-h)}{2h},
\end{align}
where $u_e(t)$ denotes the extension of $u(t)$ to $\R$ given by:
\begin{align}\label{DQ2}
u_e(t)=\begin{cases}u(0)&\text{ for }t\leq 0\vspace{.5mm},\\u(t)&\text{ for }t\in(0,T)\vspace{.5mm},\\u(T)&\text{ for }t\geq T.\end{cases}
\end{align}
The results in the following proposition have been established by Koch and Lasiecka \cite{KL}.
\begin{prop}[\cite{KL}]  \label{prop:DQ1}
Let $u\in C([0,T];X)$ where $X$ is a Hilbert space with inner product $(\cdot,\cdot)_X$.  Then,
\begin{align}\label{DQ3}
\lim_{h\ra\infty}\int_0^T(u,D_hu)_Xdt=\frac{1}{2}\left(\|u(T)\|_X^2-\|u(0)\|_X^2\right).
\end{align}
If, in addition, $u_t\in C([0,T];X)$, then
\begin{align}\label{DQ4}
\int_0^T(u_t,(D_hu)_t)_Xdt=0,\text{ for each }h>0,
\end{align}
and, as $h\ra0$,
\begin{align}\label{DQ5}
D_hu(t)\lra u_t(t)\text{ weakly in }X,\text{ for every }t\in(0,T),
\end{align}
\begin{align}\label{DQ6}
D_hu(0)\lra \frac{1}{2}u_t(0)\text{ and }D_hu(T)\lra\frac{1}{2}u_t(T)\text{ weakly in }X.
\end{align}
\end{prop}

\begin{prop}[\cite{GR}] \label{prop:DQ2} Let $X$ and $Y$ be Banach spaces.  Assume $u\in C([0,T];Y)$ and $u_t\in L^1(0,T;Y)\cap L^p(0,T;X)$, where $1\leq p<\infty$.  Then, $D_hu\in L^p(0,T;X)$ and $\|D_hu\|_{L^p(0,T;X)}\leq\|u_t\|_{L^p(0,T;X)}$.  Moreover, $D_hu\lra u_t$ in $L^p(0,T;X)$, as $h\ra 0$.
\end{prop}
\smallskip
\subsubsection{Proof of the Energy Identity.}
Throughout the proof, we fix $t\in(0,T_0]$ and let $(u,w)$ be a weak solution of the system \eqref{PDE} in the sense of Definition \ref{def:weaksln}.  Recall the regularity of $u$ and $w$.  We can define the difference quotient $D_hu(\tau)$ on $[0,t]$ as \eqref{DQ1}, i.e., $D_hu(\tau)=\frac{1}{2h}[u_e(\tau+h)-u_e(\tau-h)]$, where $u_e(\tau)$ extends $u(\tau)$ from $[0,t]$ to $\R$ as in \eqref{DQ2}. Similarly, we can define $D_h w(\tau)$ on $[0,t]$. By Proposition \ref{prop:DQ2}, with $X=Y=L^{2}(\G)$, we have
\begin{align}\label{EI2}
D_hw\in L^{2}(\G\times(0,t))\text{ and }D_hw\lra w_t\text{ in }L^{2}(\G\times(0,t)).
\end{align}
Moreover, since $u\in C([0,t];H^1_{\G_0}(\O))$ and $w\in C([0,t];H^2_0(\G))$, then 
\begin{align}\label{EI3}
D_hu\in C([0,t];H^1_{\G_0}(\O))\text{ and }D_hw\in C([0,t];H^2_0(\G)).
\end{align}
Since $u_t\in C([0,t];L^2(\O))$ and $w_t\in C([0,t];L^2(\O))$, we see that
\begin{align}\label{EI4}
(D_hu)_t\in L^1(0,t;L^2(\O))\text{ and }(D_hw)_t\in L^1(0,t;L^2(\G)).
\end{align}

Thus, \eqref{EI3}-\eqref{EI4} show that $D_hu$ and $D_hw$ satisfy the required regularity conditions to be suitable test functions in Definition \ref{def:weaksln}.  Therefore, by taking $\phi=D_hu$ in \eqref{wkslnwave} and $\psi=D_hw$ in \eqref{wkslnplt}, we obtain
\begin{align}\label{EI5}
&(u_{t}(t),D_hu(t))_\O - (u_1,D_hu(0))_\O-\int_0^t(u_t(\tau),(D_hu)_t(\tau))_\O d\tau\notag\\
&\hspace{5mm}+\int_0^t(u(\tau),D_hu(\tau))_{1,\O}d\tau-\int_0^t(w_t(\tau),\g D_hu(\tau))_\G d\tau=0,
\end{align}
and
\begin{align}\label{EI6}
&(w_t(t)+ \g u(t),D_hw(t))_\G-(w_1+ \g u(0),D_hw(0))_\G-\int_0^t(w_t(\tau),(D_hw)_t(\tau)_\G d\tau\notag\\
&\hspace{5mm}-\int_0^t(\g u(\tau),(D_hw)_t(\tau))_\G d\tau +\int_0^t(w(\tau),D_hw(\tau))_{2,\G}d\tau \notag\\
&\hspace{5mm}=\int_0^t\int_{\G}h(w(\tau))D_hw(\tau) d\G d\tau.
\end{align}
Next we must justify passing to the limit as $h\ra 0$ in \eqref{EI5} and \eqref{EI6}. Since $u, \, u_t\in C([0,t];L^2(\O))$ and $w, \, w_t\in C([0,t];L^2(\G))$, then as $h\rightarrow 0$, it follows from  (\ref{DQ6}) that
\begin{align*}
D_h u(0) & \longrightarrow \frac{1}{2}u_t(0) \text{\,\,\,and\,\,\,}
D_h u(t) \longrightarrow \frac{1}{2}u_t(t)  \text{\,\,\,weakly in\,\,\,} L^2(\O),\\
D_h w(0) &  \longrightarrow \frac{1}{2}w_t(0) \text{\,\,\,and\,\,\,}
D_h w(t) \longrightarrow \frac{1}{2}w_t(t)  \text{\,\,\,weakly in\,\,\,} L^2(\G).
\end{align*}
Therefore,
\begin{align}\label{3.11}
\begin{cases}
\lim_{h\rightarrow 0} \Big((u_{t}(t), D_hu(t))_\O- (u_1,D_hu(0))_\O\Big)= \frac{1}{2} \Big(\|u_t(t)\|_2^2 - \|u_t(0)\|_2^2\Big), \vspace{.1in} \\
\lim_{h\rightarrow 0}  (w_t(t)  + \g u(t), D_hw (t) )_\G = \frac{1}{2} |w_t(t)|_2^2 + \frac{1}{2} ( \g u(t), w_t (t) )_\G,   \vspace{.1in} \\
\lim_{h\rightarrow 0} (w_1 +\g u(0) , D_hw(0))_\G  = \frac{1}{2} |w_t(0)|_2^2 + \frac{1}{2} ( \g u(0), w_t (0) )_\G. \\
\end{cases}
\end{align}
Also, by \eqref{DQ4},
\begin{align}\label{EI10}
\int_0^t(u_t,(D_hu)_t)_\O d\tau= \int_0^t(w_t,(D_hw)_t)_\G d\tau=0.
\end{align}
In addition, since $u\in C([0,t];H^1_{\G_0}(\O))$, then \eqref{DQ3} yields
\begin{align}\label{EI11}
\lim_{h\ra 0}\int_0^t(u,D_hu)_{1,\O}d\tau=\frac{1}{2}\left(\|u(t)\|_{1,\O}^2-\|u(0)\|_{1,\O}^2\right).
\end{align}
Similarly, we obtain
\begin{align}\label{EI11.5}
\lim_{h\ra 0}\int_0^t(w,D_hw)_{2,\G}d\tau=\frac{1}{2}\left(\|w(t)\|_{2,\G}^2-\|w(0)\|_{2,\G}^2\right).
\end{align}
In order to address the plate source, we note that $w\in C([0,t];H^2_0(\G))$ and thus $h(w)\in L^{2}(\O\times(0,t))$.  It follows then from \eqref{EI2} that
\begin{align}\label{EI15}
\lim_{h\ra0}\int_0^t\int_\G h(w)D_hwd\G d\tau=\int_0^t\int_\G h(w)w_td\G d\tau.
\end{align}
The trouble terms, $\int_0^t(w_t(\tau),\g D_hu(\tau))_\G d\tau$ and $\int_0^t(\g u(\tau),(D_hw)_t(\tau))_\G d\G d\tau$, are handled as follows.  For all sufficiently small $h>0$, we have
\begin{align}\label{EI16}
&\int_0^t(\g u(\tau),(D_hw)_t(\tau))_\G d\tau\notag\\
&\hspace{5mm}=\frac{1}{2h}\left(\int_0^t(\g u(\tau),w_t(\tau+h))_\G d\tau-\int_0^t(\g u(\tau),w_t(\tau-h))_\G d\tau\right)\notag\\
&\hspace{5mm}=\frac{1}{2h}\left(\int_h^t(\g u(\tau-h),w_t(\tau))_\G d\tau-\int_0^{t-h}(\g u(\tau+h),w_t(\tau))_\G d\tau\right),
\end{align}
where we have used a change of variables in \eqref{EI16} and the fact that $w_t=0$ outside of the interval $[0,t]$.  By rearranging the terms in \eqref{EI16}, we obtain
\begin{align}\label{EI17}
&\int_0^t(\g u(\tau),(D_hw)_t(\tau))_\G d\tau=-\int_0^t(\g D_hu(\tau),w_t(\tau))_\G d\tau\notag\\
&\hspace{5mm}-\frac{1}{2h}\left(\int_0^h(\g u(\tau-h),w_t(\tau))_\G d\tau-\int_{t-h}^t(\g u(\tau+h),w_t(\tau))_\G d\tau\right).
\end{align}
We now utilize the continuity of $w_t$ in the last two terms of \eqref{EI17} as follows.
\begin{align}\label{EI18}
&\frac{1}{2h}\int_0^h(\g u(\tau-h),w_t(\tau))_\G d\tau=\frac{1}{2h}\int_0^h(\g u(0),w_t(\tau))_\G d\tau\notag\\
&\hspace{5mm}=\frac{1}{2h}\int_0^h(\g u(0),w_t(\tau)-w_t(0))_\G d\tau+\frac{1}{2h}\int_0^h(\g u(0),w_t(0))_\G d\tau\notag\\
&\hspace{5mm}\lra\frac{1}{2}(\g u(0),w_t(0))_\G,
\end{align}
as $h\ra 0$.  Similarly, we have
\begin{align}\label{EI19}
&\frac{1}{2h}\int_{t-h}^t(\g u(\tau+h),w_t(\tau))_\G d\tau=\frac{1}{2h}\int_{t-h}^t(\g u(t),w_t(\tau))_\G d\tau\notag\\
&\hspace{5mm}=\frac{1}{2h}\int_{t-h}^t(\g u(t),w_t(\tau)-w_t(t))_\G d\tau+\frac{1}{2h}\int_{t-h}^t(\g u(t),w_t(t))_\G d\tau\notag\\
&\hspace{5mm}\lra\frac{1}{2}(\g u(t),w_t(t))_\G,
\end{align}
as $h\ra0$.  Finally, by adding \eqref{EI5} and \eqref{EI6} and combining \eqref{3.11}-\eqref{EI19}, 
we can pass to the limit as $h\ra 0$ to obtain the energy identity (\ref{energy}) for weak solutions.

\vspace{0.1 in}

\subsection{Continuous Dependence on Initial Data and Uniqueness of Weak Solutions}
\noindent{}In this subsection, we prove that weak solutions continuously depend on the initial data. As a corollary, we also obtain the uniqueness of weak solutions. 

Let $U_0=(u_0,w_0,u_1,w_1)\in H$, where $H=H^1_{\G_0}(\O)\times H^2_0(\G)\times L^2(\O)\times L^2(\G)$.

Assume that $\{U^n_0=(u^n_0,w^n_0,u^n_1,w^n_1)\}$ is a sequence of initial data in $H$ that satisfies
\begin{align}\label{CDID1}
U^n_0\lra U_0\text{ in }H,\text{ as }n \ra\infty.
\end{align}
Let $\{(u^n,w^n)\}$ and $(u,w)$ be the weak solutions to \eqref{PDE} in the sense of Definition \ref{def:weaksln}, corresponding to the initial data $\{U^n_0\}$ and $\{U_0\}$, and define the quadratic energy $E^n(t)$ for a weak solution $(u^n,w^n)$ by
\begin{align}\label{CDID2}
E^n(t):=\frac{1}{2}\left(\|u^n(t)\|_{1,\O}^2+\|w^n(t)\|_{2,\G}^2+\|u^n_t(t)\|_2^2+|w^n_t(t)|_2^2\right).
\end{align}
Similarly to the work in subsection \ref{sub-local} on approximate solutions, 
we can choose a $K$ large enough so that $K^2\geq 4E(0)+1$ and $K^2 \geq 4E^n(0)+1$ for all $n$, and further establish a uniform interval of existence $[0,T_0]$ for $(u,w)$, $(u^n,w^n)$ independent of $n\in\N$, such that $E^n(t) \leq \frac{K^2}{2}$ for all $t\in [0,T_0]$.

Next define $y^n(t):=u(t)-u^n(t)$, $z^n(t):=w(t)-w^n(t)$, and
\begin{align}\label{CDID3}
\tilde{E}_n(t):=\frac{1}{2}\left(\|y^n(t)\|_{1,\O}^2+\|z^n(t)\|_{2,\G}^2+\|y^n_t(t)\|_2^2+|z^n_t(t)|_2^2\right),
\end{align}
for $t\in[0,T_0]$.  As $(u_n,w_n)$ is a weak solution satisfying the variational identities \eqref{wkslnwave} and \eqref{wkslnplt}, we have by construction that $y^n$ and $z^n$ satisfy:
\begin{align}\label{CDID4}
&(y^n_{t}(t),\phi(t))_\O- (y^n_t(0),\phi(0))_\O-\int_0^t(y^n_t(\tau),\phi_t(\tau))_\O d\tau+\int_0^t(y^n(\tau),\phi(\tau))_{1,\O}d\tau\notag\\
&\hspace{5mm}-\int_0^t(z^n_t(\tau),\g\phi(\tau))_\G d\tau=0,
\end{align}
and
\begin{align}\label{CDID5}
&(z^n_t(t),\psi(t))_\G-(z^n_t(0),\psi(0))_\G-\int_0^t(z^n_t(\tau),\psi_t(\tau))_\G d\tau+(\g y^n(t),\psi(t))_\G\notag\\
&\hspace{5mm}-(\g y^n(0),\psi(0))_\G-\int_0^t(\g y^n(\tau),\psi_t(\tau)_\G d\tau+\int_0^t(z^n(\tau),\psi(\tau))_{2,\G} d\tau\notag\\
&\hspace{5mm}=\int_0^t\int_{\G}(h(w(\tau))-h(w^n(\tau)))\psi(\tau) d\G d\tau,
\end{align}
for all $t\in[0,T_0]$ and for all test functions $\phi$ and $\psi$ as described in Definition \ref{def:weaksln}.  Let $\phi(\tau)=D_hy^n(\tau)$ in \eqref{CDID4} and $\psi(\tau)=D_hz^n(\tau)$ in \eqref{CDID5} for $\tau\in[0,t]$ where the difference quotients $D_hy^n$ and $D_hz^n$ are defined in \eqref{DQ1}.  Using a similar argument as in obtaining the energy identity for weak solutions in subsection \ref{sub-EI}, we can pass to the limit as $h\lra0$ and deduce
\begin{align}\label{CDID6}
\tilde{E}_n(t)=\tilde{E}_n(0)+\int_0^t\int_{\G}(h(w(\tau))-h(w^n(\tau)))z^n_t(\tau) d\G d\tau,
\end{align}
for all $t\in[0,T_0]$.  Since $h$ is locally Lipschitz from $H^2_0(\G)$ into $L^2(\G)$ by Proposition \ref{hprop}, then it is straightforward to obtain for each $n$ that
\begin{align}\label{CDID39}
\int_0^t\int_\G (h(w(\tau))-h(w^n(\tau)))z^n_t(\tau)d\G d\tau &\leq C(K)\left(\int_0^t\|z^n\|_{2,\G}^2d\tau\right)^{\frac{1}{2}}\left(\int_0^t|z^n_t|_2^2d\tau\right)^{\frac{1}{2}}\notag\\
&\leq C(K)\int_0^t\tilde{E}_n(\tau)d\tau.
\end{align}
Now from \eqref{CDID6} and \eqref{CDID39}, we have 
\begin{align*}
\tilde{E}_n(t)\leq\tilde{E}_n(0)+C(K)\int_0^t\tilde{E}_n(\tau)d\tau,
\end{align*}
for all $t\in[0,T_0]$.  
By Gronwall's inequality, we obtain
\begin{align}\label{CDID40}
\begin{split}
\tilde{E}_n(t)\leq \tilde{E}_n(0)e^{C(K) T_0},
\end{split}
\end{align}
for all $t\in [0,T_0]$.  From \eqref{CDID1}, we have that $\tilde{E}_n(0)\lra 0$, and thus from \eqref{CDID40} we have that $\tilde{E}_n(t)\lra0$ uniformly on $[0,T_0]$.  This concludes the proof of the continuous dependence on initial data for weak solutions.

Furthermore, by letting $U_0^n = U_0$ in (\ref{CDID1}), we obtain from (\ref{CDID40}) that $\tilde{E}_n (t)=0$ for all $t\in [0,T_0]$. Thus, weak solutions are unique.

\vspace{0.1 in}

\subsection{Global Weak Solutions}
\noindent{}In this subsection, we address the proof of global weak solutions when the source term is bounded by a linear function. We will use the strategy seen in \cite{Becklin-Rammaha2,GR,PRT-p-Laplacain} and other works.  It is the case here that either a given weak solution $(u, w)$ must exist globally in time or else one may find a value of $T_0$ with $0<T_0<\infty$, so that 
\begin{align}\label{4.1}
\limsup_{t\to T_0^-}  E(t)=+\infty,
\end{align}
where $ E(t) = \frac{1}{2}\left(\|u_t(t)\|_2^2+\|\nabla u(t)\|_2^2+|w_t(t)|_2^2+|\Delta w(t)|_2^2\right)$ is the quadratic energy. We aim to show that (\ref{4.1}) cannot happen under the assumptions $|h(s)|\leq c(|s|+1)$.  This assertion is contained in the following proposition, whose proof is identical to the proof of Proposition \ref{prop:ge} for strong solutions, as the energy identity remains valid for weak solutions.

\begin{prop}\label{prop:ge'}
Assume $h\in C^1(\mathbb R)$ such that $|h(s)|\leq c(|s|+1)$ for all $s\in\R$, where $c>0$ is a constant. Let $(u,w)$ be a weak solution of \eqref{PDE} on $[0,T]$. Then the quadratic energy $E(t)$ satisfies
\begin{align*}
E(t)\leq\left(E(0)+C t\right)e^{C t}, \text{ for all}\;\; t\in[0,T].
\end{align*}
\end{prop}

\vspace{0.1 in}

\section{Global Existence of Potential Well Solutions} \label{sec-potential}
\subsection{Local Behavior of Potential Well Solutions}
To begin our proof of Theorem \ref{thm:wellsol}, we must first establish the local behavior of potential well solutions with initial data in $\mathscr{W}_1$.  Indeed, we have the following proposition.

\begin{prop}\label{prop:W1}
Let $h$ satisfy Assumption \ref{assumption3}. Assume $(u_0,w_0)\in\mathscr{W}_1$ with $\mathscr{E}(0)<d$. Let $(u(t),w(t))$ be the local weak solution to \eqref{PDE} on $[0,T_0]$. Then $(u(t),w(t))\in \mathscr{W}_1$ for all $t\in[0,T_0]$.
\end{prop}
\begin{proof}
Let $(u(t),w(t))$ be local weak solutions pertaining to initial data $(u_0,w_0)\in\mathscr{W}_1$ with $\mathscr{E}(0)<d$.  
Note from \eqref{totalen} and \eqref{modenergy}, we have that
\begin{align}\label{potwell1}
\mathscr{J}(u(t),w(t))\leq\mathscr{E}(t)=\mathscr{E}(0)<d,
\end{align}
and thus $(u(t),w(t))\in\mathscr{W}$ for all $t\in[0,T_0]$, by \eqref{potentialwell}.

Assume by way of contradiction that there exists a time $t_1\in[0,T_0]$ such that $(u(t_1),w(t_1))\notin\mathscr{W}_1$.  But as $(u(t_1),w(t_1))\in\mathscr{W}$ and $\mathscr{W}=\mathscr{W}_1\cup\mathscr{W}_2$, we can conclude that $(u(t_1),w(t_1))\in\mathscr{W}_2$.

We next wish to show the continuity on $[0,T_0]$ of the map
\begin{align}\label{map}
t\mapsto\int_\G h(w(t)) w(t) d\G.
\end{align}
To start, note that 
\begin{align}\label{hcontinuity}
\int_\G&\Big|h(w(t))w(t)-h(w(t_0))w(t_0)\Big|d\G\notag\\
&\leq \int_\G|h(w(t))-h(w(t_0))||w(t)|d\G+\int_\G |h(w(t_0))||w(t)-w(t_0)|d\G.
\end{align}
Since $w\in C([0,T_0];H^2_0(\G))$, there exists a constant $M>0$ such that 
\begin{align}  \label{M-bd}
|w(t)|_{\infty} \leq C |\Delta w(t)|_2 \leq M, \;\;\text{for all} \;\;t\in [0,T_0].
\end{align}
Utilizing that $h\in C^1(\R)$ and the mean value theorem, as well as (\ref{M-bd}), we can write
\begin{align}\label{hcontinuity1}
\int_\G&|h(w(t))-h(w(t_0))||w(t)|d\G\notag\\
&\leq\int_\G   \max_{|s|\leq M} |h'(s)||w(t)-w(t_0)||w(t)|d\G\notag\\
&\leq C  |w(t)-w(t_0)|_2 |w(t)|_2.
\end{align}
Also with H\"{o}lder's inequality, we can write
\begin{align}\label{hcontinuity2}
\int_\G |h(w(t_0))||w(t)-w(t_0)|d\G\leq |h(w(t_0))|_2 |w(t)-w(t_0)|_2.
\end{align}
Combining \eqref{hcontinuity} and \eqref{hcontinuity1}-\eqref{hcontinuity2}, we conclude
\begin{align}\label{hcontinuity3}
\int_\G&\Big|h(w(t))w(t)-h(w(t_0))w(t_0)\Big|d\G  \leq C\left(|w(t)|_2+|h(w(t_0))|_2 \right)|w(t)-w(t_0)|_2.
\end{align}
As $w\in C([0,T_0];H^2_0(\G))$, then $w\in C([0,T_0];L^2(\G))$. Thus, we can let $t$ approach $t_0$ in \eqref{hcontinuity3} and conclude that 
\begin{align}\label{map2}
\int_\G h(w(t))w(t) d\G\lra\int_\G h(w(t_0))w(t_0) d\G
\end{align}
as $t\ra t_0$.  This yields the desired continuity of \eqref{map}. Therefore, the map
\begin{align}\label{potwell2}
t\mapsto \|\nabla u(t)\|_2^2+|\Delta w(t)|_2^2-\int_\G h(w(t))w(t) d\G
\end{align}
is continuous on $[0,T_0]$.

Considering that $(u(0),w(0))\in\mathscr{W}_1$ and $(u(t_1),w(t_1))\in\mathscr{W}_2$, we can conclude by the continuity of \eqref{potwell2} that there exists a time $s\in(0,t_1)$ such that 
\begin{align}\label{potwell3}
\|\nabla u(s)\|_2^2+|\Delta w(s)|_2^2=\int_\G h(w(s))w(s) d\G.
\end{align}

Let $t^*$ be the supremum of all $s\in (0,t_1)$ which satisfy \eqref{potwell3}.  By the continuity of \eqref{potwell2} we have $t^*\in(0,t_1)$ with $t^*$ satisfying \eqref{potwell3}, and that for any $t\in (t^*,t_1]$, $(u(t),w(t))\in\mathscr{W}_2$.  

Consider briefly the possibility of $(u(t^*),w(t^*))\neq(0,0)$.  From \eqref{nehari2} and \eqref{potwell3} we would have $(u(t^*),w(t^*))\in\mathscr{N}$, but this would imply by \eqref{depth} that $\mathscr{J}(u(t^*),w(t^*))\geq d$, contradicting \eqref{potwell1}.  Therefore we have that $(u(t^*),w(t^*))=(0,0)$.

Note now that $(u(t),w(t))\in\mathscr{W}_2$ for all $t\in(t^*,t_1]$ allows us to conclude from the definition of $\mathscr{W}_2$ that 
\begin{align}\label{potwell4}
\|\nabla u(t)\|_2^2+|\Delta w(t)|_2^2<\int_\G h(w(t))w(t)d\G
\end{align}
for all $t\in(t^*,t_1]$.  Invoking the same strategy used for lines \eqref{posd2}-\eqref{posd4} in the proof of Proposition \ref{positived}, one can quickly conclude that $\|\nabla u(t)\|_2^2+|\Delta w(t)|_2^2>s_0$ for all $t\in (t^*,t_1]$, where $s_0>0$ depends on $\G$.  As $(u,w)$ is continuous from $[0,T_0]$ into $H^1_{\G_0}(\O)\times H^2_0(\G)$, one must then have that $\|\nabla u(t^*)\|_2^2+|\Delta w(t^*)|_2^2 \geq s_0 >0$, a clear contradiction of $(u(t^*),w(t^*))=(0,0)$.  Thus by way of contradiction, $(u(t),w(t))\in\mathscr{W}_1$ for all $t\in[0,T_0]$.
\end{proof}

\subsection{Extension in Time of Potential Well Solutions}
In this final subsection we look to prove that under the hypotheses of Proposition \ref{prop:W1}, the quadratic energy $E(t)$ has a uniform bound independent of time, enabling us to invoke a standard extension procedure to establish global weak solutions. 
To this end, note from \eqref{potentialenergy} and \eqref{potwell1} that
\begin{align}\label{potwell5}
\mathscr{J}(u(t),w(t)) = \frac{1}{2}\left(\|\nabla u(t)\|_2^2+|\Delta w(t)|_2^2\right)-\int_\G H(w(t)) d\G<d,
\end{align}
for all $t\in [0,T_0]$.  As $(u(t),w(t))\in\mathscr{W}_1$ for all $t\in[0,T_0]$, we further have from condition (1) of Assumption \ref{assumption3} that
\begin{align}\label{potwell6}
\|\nabla u(t)\|_2^2+|\Delta w(t)|_2^2>\theta\int_\G H(w(t))d\G,
\end{align}
when $(u,w) \not=(0,0)$. It follows from \eqref{potwell5} and \eqref{potwell6} that
\begin{align}\label{potwell7}
\int_\G H(w(t))d\G<\frac{2d}{\theta-2}.
\end{align}
Substituting \eqref{potwell7} into \eqref{totalen} and using \eqref{modenergy}, we obtain
\begin{align}\label{potwell8}
    E(t)<\mathscr{E}(0)+\frac{2d}{\theta-2}<d+\frac{2d}{\theta-2}=\frac{\theta d}{\theta-2},
\end{align}
for all $t\in [0,T_0]$.  As the desired uniform bound (\ref{potwell8}) has been found, we can indeed extend the weak solution of $(u,w)$ in time indefinitely, finishing the proof of Theorem \ref{thm:wellsol}.

Moreover, by assuming additional regularity of the initial data, i.e., 
$U_0 = (u_0, w_0, u_1, w_1) \in \mathcal{D}(\mathscr A)$,
we can establish the existence of a global strong solution in the potential well.  
Due to (\ref{LLS9}), for strong solutions, the local existence time $T_0$ depends on $K$, where we let
\[
K^2 = \frac{4\theta d}{\theta-2} + 1 > 4E(0) + 1,
\]
in view of \eqref{potwell8}.  
Furthermore, by Theorem \ref{thm:strong}, we have 
$U(t) \in \mathcal{D}(\mathscr A)$, for all $t \in [0,T_0]$,
and by \eqref{potwell8}, $K^2 = \frac{4\theta d}{\theta-2} + 1 > 4E(T_0) + 1$.
Thus, we can extend the local strong solution from $[0,T_0]$ to $[T_0,2T_0]$.  
Iterating this procedure yields the global strong solution in the potential well.
This completes the proof of Theorem \ref{thm:wellsol2}.

\vspace{0.1 in}

\bibliographystyle{abbrv}
\bibliography{mohbib-new}

\end{document}